\documentclass[12pt]{amsart}
\usepackage{graphicx,amssymb}

\newcommand{\Lax}{\mathcal{L}}

\newcommand{\R}{\mathbb R}
\newcommand{\N}{\mathbb N}

\newtheorem{thm}{Theorem}
\newtheorem{lem}[thm]{Lemma}

\newtheorem{prop}[thm]{Proposition}
\newtheorem{princ}[thm]{Principle}

\theoremstyle{definition}

\theoremstyle{remark}

\begin{document}

\title[Dynamical Analysis of a repeated game]
{Dynamical Analysis of a repeated game with incomplete information}
\author{Xavier Bressaud}
\author{Anthony Quas}
\begin{abstract}
We study a two player repeated zero-sum game
with asymmetric information introduced by Renault
in which the underlying state of the 
game undergoes Markov evolution (parameterized by a transition
probability $\frac 12\le p\le 1$). H\"orner, Rosenberg, Solan and
Vieille identified an optimal strategy, $\sigma^*$ for the informed player 
for $p$ in the range $[\frac 12,\frac 23]$. We extend the range on which 
$\sigma^*$ is proved to be optimal to about $[\frac 12,0.719]$ 
and prove that it fails to be optimal at a value around 0.7328.
Our techniques make use of tools from dynamical systems,
specifically the notion of pressure, introduced by D. Ruelle.
\end{abstract}
\maketitle

We study a simple two player dynamic zero-sum game with asymmetric information 
introduced by Renault in \cite{Renault} 
and studied by H\"orner, Rosenberg, Solan and Vieille in \cite{HRSV}.
The system is in a state unknown to one of the players. Unlike the 
Aumann--Maschler model \cite{AumannMaschler}, the state here undergoes 
Markov evolution independent of the actions of the players.

At each stage, the system is in one of two states $S_0$ and $S_1$.
The two players, Ian and Una (for informed and uninformed respectively),
simultaneously make a choice of playing 0 or 1. If the symbols all coincide 
(that is the system is in state $S_0$ and both Ian and Una play 0; or
the system is in state $S_1$ and Ian and Una both play 1)
then Una gives Ian \$1. Otherwise no money is transferred. 

A crucial aspect of the game is that
Ian is aware of the state \emph{before} choosing his move, whereas
Una is never told of the state. Also, the money that Una pays Ian is
not paid immediately, but only after a large number of rounds of the
game have been played.
Each player sees the moves of the other, but is not informed of the 
payoff at the time (although Ian can deduce this information from what is known 
to him, whereas Una cannot).

The state of the system is assumed to undergo Markov evolution, 
where the system stays in its current state
between moves with fixed probability $p\ge \frac12$,
or switches with probability $1-p$. The transition probability governing
the switching is known to both players. We assume that the system is
initially in a random state with uniform probability. 

Ian thus faces a tradeoff between short term (he has sufficient information to
optimize his expected payoff in the current turn) versus long term (if he always plays
so as to optimize his payoff in the current turn, then he reveals the 
current state of the system
to Una, who can then use this information to minimize Ian's payoff).

The existence of a uniform value, its characterization and the existence of
optimal strategies for Una was obtained by Renault \cite{Renault}.
Neyman \cite{Neyman} extended these results to the case of partial monitoring
of the past moves, and established the existence of
optimal strategies for both players. 
That is, strategies $\sigma$
for Ian and $\tau$ for Una, such that whenever Una uses strategy $\tau$, 
Ian's long-term average expected payoff is at most $v$; whereas whenever Ian
uses strategy $\sigma$, his long-term average expected payoff is at least $v$.
Thus any strategy for Ian gives a lower bound for the value of 
the game (by taking the infimum of the expected long-term gain over all
possible counter-strategies by Una). Similarly any strategy for Una
gives an upper bound for the value of the game.

As usual in game theory, the best strategies are often mixed strategies. 
That is, given all of the information available to a player, his strategy
returns a probability vector distributing mass to the available moves.
Since we use dynamical systems theory, it is convenient to have
a compact space describing past moves that is mapped \emph{into itself}
when it is updated by recording a new move. We therefore use
the following spaces to describe the state prior to the current turn.
Let $M_I=\{0,1\}$, $M_U=\{0,1\}$ and $\mathcal S=\{S_0,S_1\}$
represent Ian's possible moves, Una's possible moves and the system's state...
A strategy for Ian can then be formally described as a map $\sigma$ from 
$\bigcup_{n\ge 0}(M_I\times M_U\times \mathcal S)^n$ to $[0,1]^2$, 
where the vector $\sigma(x,y,z)=(p_0,p_1)$
describes Ian's probabilities of playing 1 if the current state is $S_0$ or
$S_1$ respectively when Ian's past moves were $x$, Una's past moves were $y$ and
the sequence of past states is $z$.
Similarly, a strategy for Una is a map $\tau$ from
$\bigcup_{n\ge 0}(M_I\times M_U)^n$ to $[0,1]$, where $\tau(x,y)$ gives the probability of 
playing 1 if Ian's past moves were $x$ and Una's past moves were $y$.

Our goal, of course,
is essentially to find $v$ and the optimal strategies $\sigma$ and $\tau$. 
These, as one expects, depend significantly on $p$. The answer for $p=\frac12$ 
is straightforward: Ian always plays as if 
he were facing a one-shot game and wins with probability $\frac12$. The case
$p=1$ (so that the system always remains in the same state, which we assume
to be randomized uniformly)
was studied by Aumann and Maschler \cite{AumannMaschler}, where it shown that
he cannot use his information and has to play randomly as if he did 
not have any advantage (the \emph{non-revealing strategy})
and wins only with probability $\frac14$.  In \cite{HRSV}, 
the authors exhibit a strategy $\sigma^*$ for Ian (defined properly in 
Section \ref{sec:sigmastar}) and prove that it is optimal for all 
$\frac12 \leq p \leq \frac23$. In this setting, they give a simple closed 
formula, $v_p=\frac{p}{4p-1}$, for the value $v_p$ of the game 
and also provide an optimal strategy 
$\tau^*$ for Una (based on a two state automaton). 
They express the long-term payoff of the strategy $\sigma^*$ 
as the sum of a series for all 
values of the parameter, hence providing a lower bound for the value of the game 
(an alternative lower bound that is better in some regimes is given by the
trivial strategy with a bound of $\frac14$),
while an upper bound is given by the payoff 
of the strategy $\tau^*$. They compute this lower bound explicitly for specific 
values of the parameter $p$ larger than $\frac23$. In the very special case $p=p^*$ 
solving $9x^3-13x^2+6x - 1 = 0$ ($p^* \simeq 0.7589$), they observe that  
$\sigma^*$ is still optimal.  In this case they also exhibit an optimal 
strategy for Una (more tricky but still based on a finite automaton). 
Finally, they raise the question of the optimality of $\sigma^*$ for instance at 
$p=\frac34$. We provide a negative answer and prove: 
\begin{thm}
\label{th:result}
The strategy $\sigma^*$ is optimal for $p<0.719$
and not optimal for some  $p<0.733$.  
\end{thm}
\begin{figure}
\label{fig:parameter}
\includegraphics[width=10cm]{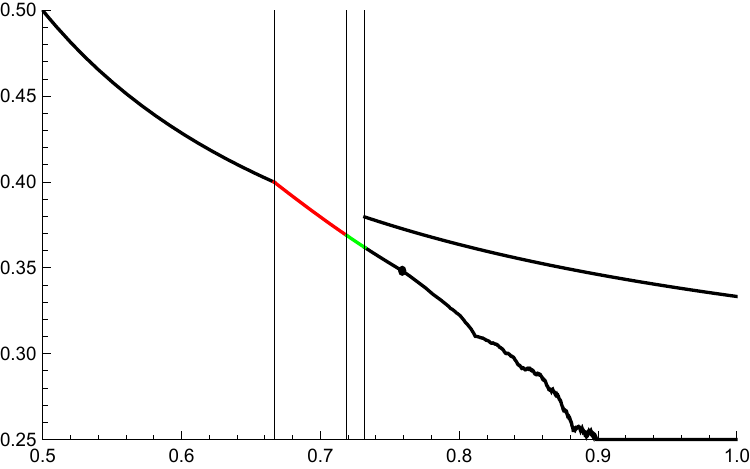}
\caption{Bounds on the value of the game as a function of the parameter 
$p$: The lower curve is the long-term payoff of $\sigma^*$. \cite{HRSV}
proved this was the value of the game in the range $[\frac 12,\frac 23]$. 
We prove this remains true up to 0.719 (grey), and give evidence that this hold
up to 0.732 (light grey). We show $\sigma^*$ is not optimal at $p=0.73275300915$.
Beyond 0.732, the upper bound for the value (top line) of the game was 
obtained in \cite{HRSV} based on a simple strategy for Una. They also found 
a particular value  $p\approx 0.7589$ for which $\sigma^*$ is optimal.}

\end{figure}
Defining $p_c=\sup\{t\colon \sigma^*\text{ is optimal for all 
$p\in[\frac 12,t]$}\}$, the theorem states that $0.719\le p_c<0.733$. 
For both the upper and lower bounds, the
proofs are based on checking that a certain finite set of inequalities is satisfied. 

The fact that $p_c\ge \frac 23$ was established in \cite{HRSV}. 
Experimentation strongly suggests that $p_c>0.732$, but we have not been able to 
show this rigorously. The methods in this article give, for each $n$, a family
($C_n$) of finitely checkable inequalities, such that if $p$ satisfies 
($C_n$) for some $n$, then $\sigma^*$ is optimal for $p$. 
The proof that $\sigma^*$ is optimal up to
0.719 proceeds by considering two intervals of parameters
and showing that on both intervals, ($C_9$) is satisfied
for all parameters in the interval. Further, if one picks values of
$p$ randomly in the range $[0.719,0.732]$ and then tests ($C_n$) for $n=50$, $n=100$,
$\ldots$, $n=500$, an experiment showed that for each of 10000
randomly selected $p$ values,
at least one of the collections of sufficient conditions for optimality of
$\sigma^*$ was satisfied. It seems likely that for any $p_0<p_c$, there is an $n$
such that $C(n)$ is satisfied by all $p\in[\frac23,p_0]$.
Unsurprisingly the first value of $n$ for which the collection of inequalities is
satisfied becomes larger as $p$ approaches the conjectured $p_c\in(0.732,0.733)$
and at the same time, the number of intervals of $p$ into which the range must
be sub-divided
is expected to grow exponentially with $n$. We are confident that one can 
go beyond $p=0.719$, but continuation requires an increasing amount of effort for
a decreasing amount of improvement.

We conjecture that $p_c$ is sharp in the sense that for $p>p_c$, 
$\sigma^*$ would in general not be optimal.  ``In general'', because 
as already pointed out, \cite{HRSV} shows there are still 
special values beyond $p_c$ at which $\sigma^*$ is optimal.
Quite surprisingly, we had to introduce tools 
from dynamical systems (thermodynamic formalism) to show the optimality of 
$\sigma^*$. The strategy of Una to which $\sigma^*$ is the optimal 
response turns out to be a strategy that takes into account
the past moves of Ian since the last `reset' (time at which Ian's 
move made it possible to deduce the current state with certainty).
Since the time since the last reset may be unbounded, we control the 
behaviour of the orbit of a certain dynamical system. However, the 
result relies, above all, on standard tools of game theory. 

The paper is laid out as follows: in Section \ref{sec:tools}, we 
introduce classical tools from game theory. In Section \ref{sec:sigmastar}, 
we define the strategy $\sigma^*$, prove its basic properties and 
compute its payoff for all $p$. In Section \ref{sec:player2}, we search 
for an optimal strategy for Una.
We give a system of equations whose solutions yield potential strategies for 
Una in the range $p\le 0.78$. Such a solution yields a desired strategy
only if it satisfies a set of inequalities. In Section 
\ref{sec:monotonicity}, we find a sufficient condition for 
these inequalities to hold in terms of the pressure of a potential. 
We show, in Section \ref{sec:pressure}, that the pressure condition is 
satisfied for all $p$ less than 0.719023. This ends the proof of the 
first part of Theorem \ref{th:result}. In Section \ref{sec:better} we 
exhibit a strategy for Ian with a larger long-term expected payoff than 
$\sigma^*$ 
for certain values of $p$; the smallest such value of $p$ that we found 
is smaller than 0.733. This  will finish the proof of Theorem \ref{th:result}. 
A final section addresses the question of which features of the game
make it amenable to an analysis of this type. 

\section{Tools from game theory}
\label{sec:tools}
The technical framework that we use to prove these statements is the study of 
Markov Decision Processes (MDP). A Markov decision process is one in 
which the system moves
around a compact state space $\Omega$, influenced by an agent who
can, at each step, choose from one of a compact (in our case, finite) set of 
transition probabilities on the state space, each one with a given one-step payoff. 
The value of the process is the maximal long-term expected value of the gain.

More formally, given a repeated game, we let $\gamma_N(\sigma,\tau)$ be the expected 
payoff per round to Player 1 if Player 1 plays the strategy $\sigma$ and 
Player 2 plays the strategy $\tau$ for $N$ rounds. 
Suppose there exists a $v\in\R$ such that 
for each $\epsilon>0$, there exists $N_0\in\N$ and a pair of strategies
$\sigma^*$ and $\tau^*$ for Players 1 and 2 respectively such that
for all $N\ge N_0$, 
\begin{align*}
&\gamma_N(\sigma^*,\tau)>v-\epsilon\text{ for each strategy $\tau$ for Player 2;}\\
&\gamma_N(\sigma,\tau^*)<v+\epsilon\text{ for each strategy $\sigma$ for Player 1.}
\end{align*}
Then $v$ is the \emph{value} of the game. 

If a game has value $v$ and there exists a strategy $\sigma^*$ such that
$\liminf_{N\to\infty} \gamma_N(\sigma^*,\tau)\ge v$ for each strategy $\tau$
for Player 2, then $\sigma^*$ is said to be an \emph{optimal strategy} for 
Player 1. Similarly if $\tau^*$ is such that $\limsup_{N\to\infty}\gamma_N
(\sigma,\tau^*)\le v$ for each strategy $\sigma$ for Player 1, then
$\tau^*$ is optimal for Player 2. 

We use the following theorem to characterize the value of a game
and optimal strategies
\begin{thm}[Average Cost Optimality Equality \cite{FeinbergShwartz}]
\label{thm:ACOE}
Suppose a Markov decision process has compact state space $\Omega$,
a compact action set $\mathcal A$, a continuous payoff function $r\colon 
\Omega\times\mathcal A\to\R$ and a continuous transition rule $q\colon \Omega\times
\mathcal A\to \mathcal P(\Omega)$ such that $q_{\omega,a}$ is a finitely supported
probability measure on $\Omega$ for each $\omega\in\Omega$ and $a\in\mathcal A$.

Suppose there exist $v\in\R$ 
and a bounded function 
$V\colon\Omega\to\R$ such that the following equation is satisfied:
\begin{equation}
V(\omega)+v = \max_{a \in\mathcal A}\left(r(\omega,a)+\int V(\omega')
\,dq_{\omega,a}(\omega')\right).
\label{eq:MDP}
\end{equation}

Then $v$ is the value of the Markov decision process for each initial state $\omega$. 
Further, a stationary strategy $\alpha\colon\Omega\mapsto\mathcal A$ is optimal
if $\alpha(\omega)$ attains the maximum in the right side of \eqref{eq:MDP}
for each $\omega\in\Omega$.
\end{thm}


We interpret $V(\omega)$ as the relative score of the position $\omega$. This is there
in order to take long-term effects into account. This can be thought of as
answering the question \emph{What is the long-term total 
difference between starting at some fixed $\omega_0$ and starting at $\omega$?}
This will be finite under suitable continuity and contractivity
assumptions. The equation \eqref{eq:MDP} 
informally says that if one chooses the action $a$ 
achieving the maximum, then the expected gain plus difference in $V$ values is $v$. 

The way we use Theorem \ref{thm:ACOE} is as follows. Suppose (for example) Una is
looking for a best response to a strategy $\sigma$ for Ian that is based upon the current
state of the system as well as Una's current \emph{belief} that the system is in state 1 
(that is the conditional probability that the system is in state 1 given the information 
available to her).
We let the state space be $\Omega=[0,1]$, the space of beliefs.
Una's belief is initially $\frac 12$ and is updated after each move.

Let us suppose that $v\in\R$ and $V\colon \Omega\to\R$ satisfy \eqref{eq:MDP}.
Una is then trying to decide between playing 0 and 1.
Since she knows $\omega$, she has computed the probability that the system is
in state $S_0$ or $S_1$, and can also compute the probability that Ian will play 0 or 1. 
Hence she can compute the expected one-round payoff to Ian
if she plays either 0 or 1.
An best response (there may be many) to $\sigma$ is
any strategy that always picks an option attaining the minimum expectation
of (payoff + $V$).

We now turn to another frequently used idea in zero-sum games:
\begin{princ}\label{princ:bestbest}
Suppose that 
\begin{enumerate}
\item $\tau$ is a best response to $\sigma$; and
\item $\sigma$ is a best response to $\tau$
\end{enumerate}
Then $\sigma$ is an optimal strategy for Ian.
Similarly $\tau$ is an optimal strategy for Una.
\end{princ}
See for example \cite{HRSV}.
We exploit this principle repeatedly in the remainder of this article.

A symmetry argument explained in \cite{HRSV} shows that for 
$0\leq p \leq \frac12$, $v_p = v_{1-p}$. Hence, in what follows
we consider the case $\frac12\le p\le 1$.  We will be looking
mainly at the strategy $\sigma^*$ introduced in \cite{HRSV}.

In what follows, if Ian is assumed to be playing using the strategy $\sigma^*$
(to be defined below), we frequently refer to Una's \emph{belief} that the system
is in state $S_1$. Formally, this is just the 
conditional probability that the system is in the state $S_1$ given all 
the information
available to Una (that is the sequence of past moves made by both players),
given that Ian is using $\sigma^*$. 
Of course, Ian can calculate Una's belief that the system is in state $S_1$. 

\section{The strategy $\sigma^*$}
\label{sec:sigmastar}

We now describe a strategy, $\sigma^*$, that we show to be optimal for
Ian for a range of the parameter. This strategy was initially introduced
in \cite{HRSV}. As pointed out below, it is characterized by being a
greedy U-indifferent strategy. 

We define two maps as follows: 
\begin{align*}
f_0(\theta)&=\begin{cases}
p\frac{2\theta-1}\theta+(1-p)\frac{1-\theta}{\theta}&\text{if $\theta\ge\tfrac12$};\\
1-p&\text{if $\theta\le\tfrac12$}.
\end{cases}\\
f_1(\theta)&=\begin{cases}
p&\text{if $\theta\ge \tfrac12$};\\
p\frac{\theta}{1-\theta}+(1-p)\frac{1-2\theta}{1-\theta}&\text{if $\theta\le\tfrac12$}.
\end{cases}
\end{align*}
Notice that $f_0(\theta)=1-f_1(1-\theta)$. We define a function $\Phi$ by setting 
$\Phi(x)$ to be $f_0(x)$ if $x\ge \frac12$ and $f_1(x)$ otherwise. We set 
$p_n = \Phi^n(p)$ for all $n \geq 0$. 

\begin{figure}
\includegraphics[width=6cm]{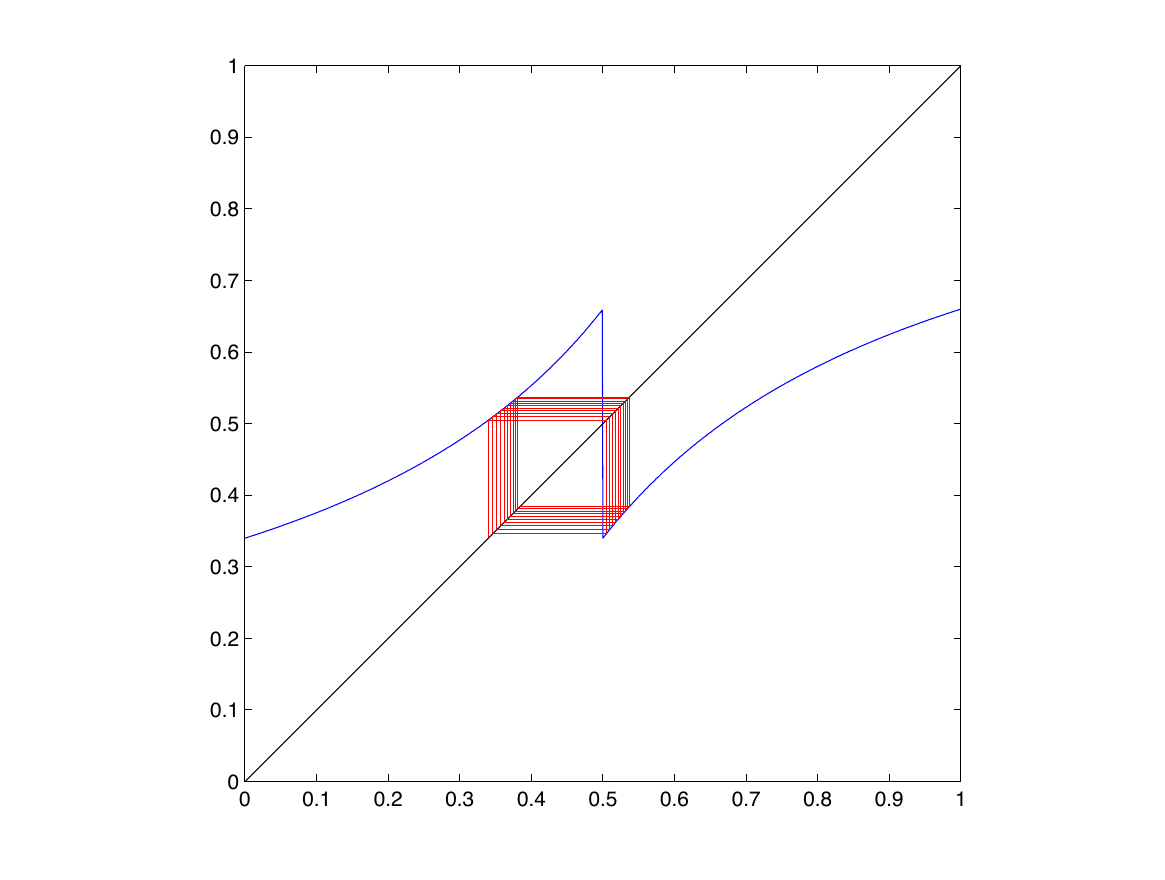}
\includegraphics[width=6cm]{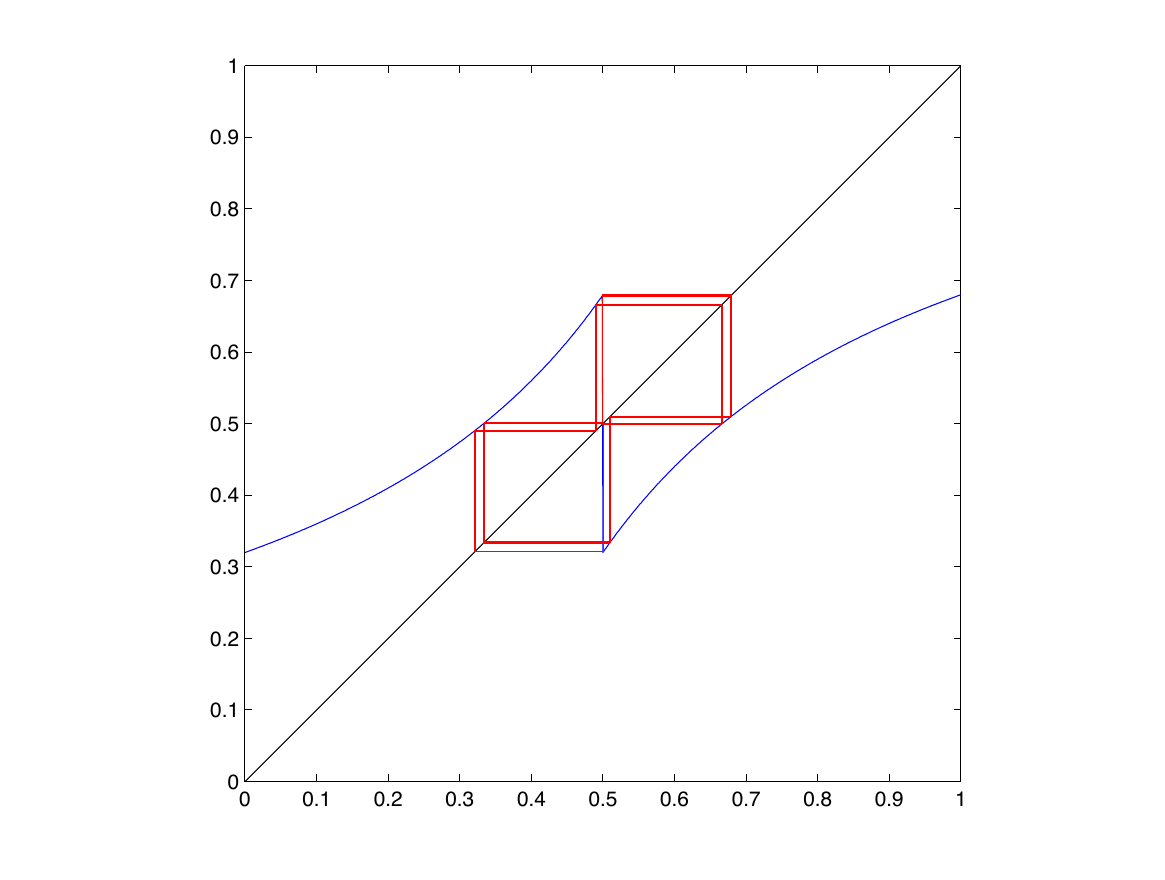}
\includegraphics[width=6cm]{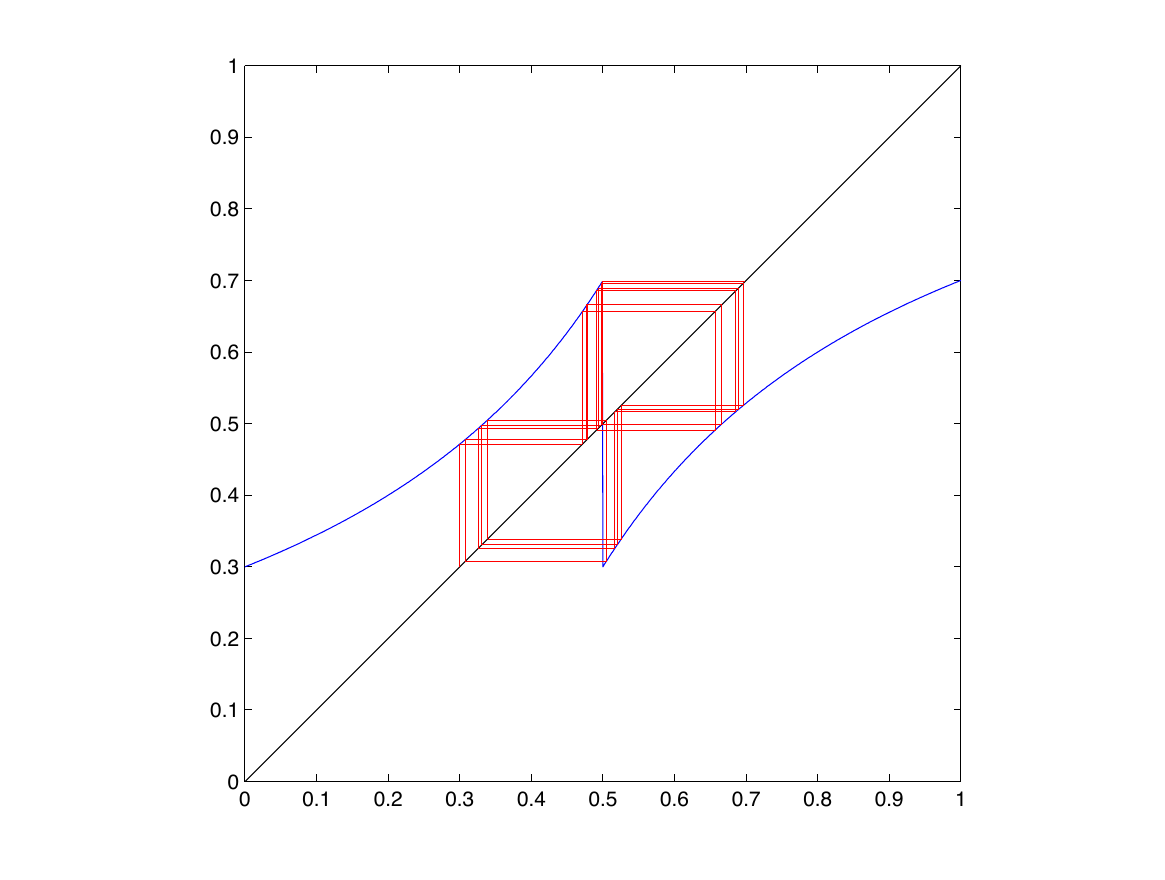}
\includegraphics[width=6cm]{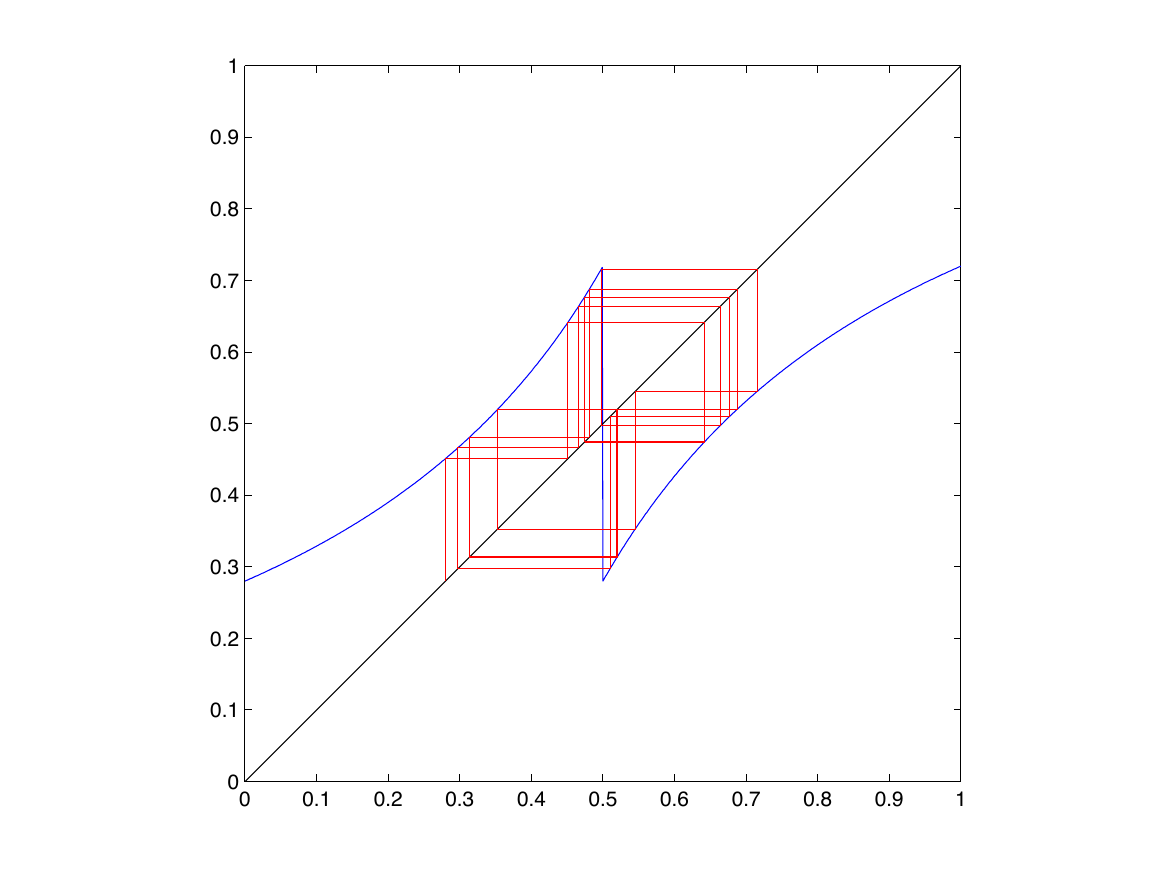}
\caption{The graphs of $\theta \mapsto \Phi(\theta)$ and first points of 
the orbit of $1-p$, for values of $p$ ranging from $p=0.66$ to $p=0.72$ 
in steps of $0.02$.} 
\end{figure}

The strategy $\sigma^*$ is then defined as follows.
Ian computes Una's belief, $\theta$, that the system is in state 1.
He then plays 1 with the following probabilities:
\begin{equation*}
\mathbb P(\text{playing 1})=\begin{cases}
1&\text{if the system is in state $S_1$ and $\theta\le \tfrac12$};\\
\frac{1-2\theta}{1-\theta}&\text{if the system is in state $S_0$ and 
$\theta\le\tfrac12$};\\
\frac{1-\theta}{\theta}&\text{if the system is in state $S_1$ 
and $\theta\ge\tfrac12$};\\
0&\text{if the system is in state $S_0$ and $\theta\ge\tfrac12$.}
\end{cases}
\end{equation*}
He plays 0 with the complementary probability. 

As shown in \cite{HRSV}, the maps $f_0$ and $f_1$ keep track of Una's belief that
the system is in state $S_1$ if Una knows that Ian is playing $\sigma^*$ by Bayesian
updating. For example if Una's belief that the system is in state $S_1$ 
is $\theta>\frac 12$
then Una attaches probabilities $\theta\cdot \frac{1-\theta}{\theta}=1-\theta$
to $(S_1,1)$, $\theta\cdot \frac{2\theta-1}{\theta}=2\theta-1$ to $(S_1,0)$
and $1-\theta$ to $(S_0,0)$,
where $(S_i,j)$ means the event that the system is in state $S_i$ and Ian plays $j$.
If Ian plays 0, Una computes the probabilities of the system having been in $S_1$
to be $(2\theta-1)/((2\theta-1)+(1-\theta))=(2\theta-1)/\theta$, so that
her updated belief that the system is in state $S_1$ is
$\frac{2\theta-1}\theta p + \frac{1-\theta}{\theta}(1-p)=f_0(\theta)$. 

The critical feature of $\sigma^*$ that we make use of is the fact that
the expected long-term average gain for Ian if he plays $\sigma^*$ is the same 
no matter which strategy is used by Una. We prove this in the lemma below.
In view of this lemma and Principle \ref{princ:bestbest}, 
if one can find a strategy $\tau$ for Una, to which $\sigma^*$
is a best response, then $\sigma^*$ and $\tau$ are optimal strategies
for Ian and Una respectively.

\begin{lem}
\label{lem:payoff}
The expected long-term average gain for Ian when playing strategy $\sigma^*$ 
is independent of the strategy played by Una. Hence any strategy $\tau$ for Una
is a best response to $\sigma^*$.
\end{lem}

\begin{proof}
We consider the Markov decision process for Una. The state of the process will be
just her belief, $\theta$, that the system is in the state $S_1$. 
Her action has no effect on the evolution of the state, and so her chosen move will
just be the one with the lower expected one-stage payoff. 

Suppose without loss of generality that $\theta\ge \frac12$. 
Then if Una plays 0, then Ian gains if the system was in state $S_0$ (if
$\theta\ge\frac12$ then Ian always plays 0 if the system is in state $S_0$). The
expected one-step gain for Ian from this strategy is therefore $1-\theta$.
Similarly, if Una plays 1, then Ian gains if the system was in state $S_1$
and Ian chose to play 1. This happens with probability $\theta\times (1-\theta)/
\theta=1-\theta$. 

Similarly, if $\theta<\frac12$, the expected one-step gain for Ian is $\theta$, 
independently of any move played by Una.

Hence the expected one-step gain from any position does not depend on Una's move.
The next position attained by the system is also independent of Una's move.
So the long-term average gain is also independent of Una's choice of moves and
Ian's long-term average gain is independent of Una's strategy.
\end{proof}

We call a strategy for Ian with the property in the lemma above \emph{U-indifferent}.
A strategy is U-indifferent if the probabilities (given Una's information) that
the system is in state $S_1$ and Ian plays 1 and that the system is in state $S_0$ and 
Ian plays 0 are equal. This probability is then the expected one-step
gain for Ian. 
In fact, $\sigma^*$ is the greedy U-indifferent strategy: the expected 
one-step gain is $\min(\theta,1-\theta)$ as shown above.
On the other hand, if Ian is playing any strategy and Una's belief that the 
system is in state $S_1$ is $\theta$, then the minimum of the probabilities
that the system is in state $S_1$ and Ian plays 1 and that the system is in
state $S_0$ and Ian plays 0 is at most $\min(\theta,1-\theta)$. Hence
Una can ensure that Ian's expected one-step gain  is at most $\min(\theta,1-\theta)$. 
This quantity is maximized by $\sigma^*$.

Consider the evolution of Una's beliefs. In all stages after the first, 
these
belong to the set $\bigcup_{n\ge 0}\Phi^n\{p,1-p\}$. 
Notice that the values of $f_0(x)$ and $f_1(x)$ depend on $p$, but
we suppress the dependence on $p$ from the notation since $p$ is fixed.
Since for $x\ge \frac12$, we have $f_0(1-x)=1-f_1(x)$, we have
$\Phi^n(1-p)=1-\Phi^n(p)$ for all $n$. 

When $\theta\ge\frac12$, the belief returns to $p$ when the system is in state 
$S_1$ and Ian plays $1$. If $\theta>\frac 12$ and the system is in state
$S_0$ (i.e. there is a mismatch between Una's belief and the state of the system),
Ian never selects 1. When $\theta\le\frac12$, the belief returns to $1-p$ 
when the system is in state $S_0$ and Ian selects 0. 

We view this as a ladder (see Figure \ref{fig:ladder}) with base $\{p , 1-p\}$ and rungs 
$\{p_n, 1-p_n\}$, for $n\geq 1$,
on which the belief follows a Markov chain: at each step, one either ascends 
one level, or
falls down to the base. Falling off corresponds to making the choice that 
returns the state to $p$ or $1-p$. 

\begin{lem}
If Ian plays strategy $\sigma^*$, then his long-term expected gain is equal to 
the proportion of time spent at the base of the ladder, 
irrespective of the strategy played by Una. 

We can therefore deduce an explicit lower bound (in the form of an infinite sum) 
for the value of the game as a function of the parameter $p$.
\end{lem}

\begin{proof}
Consider the evolution of Una's beliefs. These
always belong to the set $\bigcup_{n\ge 0}\Phi^n\{p,1-p\}$. 

Recall from Lemma \ref{lem:payoff} that  if Una's belief is $\theta$,  
the one-step expected payoff for Ian is given by $\min(\theta,1-\theta)$ 
independently of the strategy played by Una. 

On the other hand, the probability of returning to $p$ or $1-p$ from
$\theta$ or $1-\theta$  is also $\min(\theta,1-\theta)$. We verify this in the 
case $\theta\ge\frac12$. The belief returns to $p$ only if the system is in state 
$S_1$ and Ian selects $1$. The probability of this is 
$1-\theta=\min(\theta,1-\theta)$ as required.

Hence from the $n$th rung of the ladder, the probability of falling off  
is $\min(\Phi^n(p),1-\Phi^n(p))$. This is the same as the expected payoff 
from that state. That is, in any position, the expected payoff from the next turn 
is equal to the probability of falling off the ladder at the next turn.
We let $u_n=\max(\Phi^n(p),1-\Phi^n(p))$ be the complementary probability:
the probability of continuing up the ladder from the $n$th stage.

\begin{figure}
\includegraphics{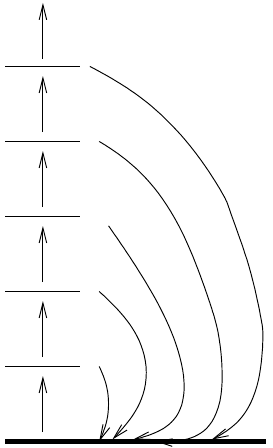}
\caption{Una's belief that the system is in state $S_1$
can be modeled by a ladder: if Ian plays 0 while $\theta<\frac12$;
or 1 while $\theta>\frac12$, then the belief becomes $1-p$ or $p$
respectively, corresponding to the bottom rung of the ladder. Note that the $n$th
rung of the ladder corresponds both to $\Phi^n(p)$ and $\Phi^n(1-p)$}
\label{fig:ladder}
\end{figure}

One can check that for this Markov chain, the stationary distribution
gives level $n$ probability
$$
\pi_n=\frac{u_0\ldots u_{n-1}}
{1+u_0+u_0u_1+u_0u_1u_2+\ldots}.
$$
We do not specify any initial measure, but the renewal structure 
of the chain shows that on the long term the gain is described by 
the invariant measure, independently of the initial conditions: 
after a random but finite amount of time, Ian will play so 
that $\theta$ becomes $p$ (or $1-p$). 

Since in any state, the expected gain is the same as the probability of 
`falling off the ladder', we see that the expected gain per round for Ian if 
he plays $\sigma^*$ is 
given by
$$
Y=\frac{1}{1+u_0+u_0u_1+u_0u_1u_2+\ldots},
$$
irrespective of Una's strategy,
where we recall that the quantities $(u_i)_{i\ge 0}$ are functions of $p$. 
We observe that this expression was already derived in \cite{HRSV}. 
\end{proof}

This $Y$ is a lower bound for the value of the game. We give an 
alternative expression for $Y$ in terms of a sum of matrix products. 
This is not strictly necessary for what follows, but it is here as
we think it will help the reader gain a better understanding.
This expression should be compared with the expression that arises later for $1/v$
($v$ being the value of the game in some ranges of $p$). 

We will write $p_n=\Phi^n(p)$ as a quotient of two polynomials in $p$: $p_n=
a_n/b_n$, so that $p_0=p/1$. Also write $\epsilon_n=1$ if 
$p_n\ge\frac12$ and 0 otherwise.

If $\epsilon_n=1$, we have $p_{n+1}=f_0(p_n)$, while if $\epsilon_n=0$, we have
$p_{n+1}=f_1(p_n)$.

If $\epsilon_n=1$, we have $u_n=p_n=a_n/b_n$ and 
\begin{align*}
\frac{a_{n+1}}{b_{n+1}}&=f_0(a_n/b_n)=
\frac{p(2a_n-b_n)/b_n+(1-p)(b_n-a_n)/b_n}{a_n/b_n}\\
&=\frac{a_n(3p-1)-b_n(2p-1)}{1a_n+0b_n}.
\end{align*}

Similarly if $\epsilon_n=0$, we have $u_n=1-p_n=(b_n-a_n)/b_n$ and
\begin{align*}
\frac{a_{n+1}}{b_{n+1}}&=f_1(a_n/b_n)=
\frac{pa_n/b_n+(1-p)(b_n-2a_n)/b_n}{(b_n-a_n)/b_n}\\
&=\frac{(3p-2)a_n+(1-p)b_n}{-a_n+b_n}.
\end{align*}

In both cases, we see that $u_n=b_{n+1}/b_n$.
Introducing matrices $U_1=\begin{pmatrix}3p-1&-(2p-1)\\1&0\end{pmatrix}$ and 
$U_0=\begin{pmatrix}3p-2&1-p\\-1&1\end{pmatrix}$, we have
\begin{equation*}
\begin{pmatrix}a_{n+1}\\b_{n+1}\end{pmatrix}
=U_{\epsilon_n}
\begin{pmatrix}a_{n}\\b_{n}\end{pmatrix}.
\end{equation*}

Now, taking the product of the $u_n$'s, we obtain \emph{par
t\'el\'escopage} $u_0\cdots u_n=b_{n+1}/b_0=b_{n+1}$. Hence we get the expression
\begin{equation*}
u_0u_1\cdots u_n=b_{n+1}=\begin{pmatrix}0&1\end{pmatrix}
U_{\epsilon_n}\ldots U_{\epsilon_0}
\begin{pmatrix}p\\1\end{pmatrix}.
\end{equation*}

Summing over $n$, we obtain another expression for 
the average long-term gain that will accrue to Ian if he plays $\sigma^*$.

\begin{equation}\label{eq:vval1}
\frac1Y=\begin{pmatrix}0&1\end{pmatrix}
(I+U_{\epsilon_0}+U_{\epsilon_1}U_{\epsilon_0}+
U_{\epsilon_2}U_{\epsilon_1}U_{\epsilon_0}+\ldots)\begin{pmatrix}p\\1\end{pmatrix}
\end{equation}

\section{Strategies for Una}
\label{sec:player2}
In \cite{HRSV}, the authors showed that $\sigma^*$ is optimal for
$p\in [\frac12,\frac23]$ and for a specific $p^*\approx 0.7589$
that is the unique value of $p$ for which $p_1>\frac 12$ and
$p_1=1-p_2$.
In both cases, they exhibit a strategy for Una based on a finite state automaton 
where transitions in the automaton are governed by actions of Ian and 
then show that $\sigma^*$ is a best response to this strategy.  For 
$p > \frac23$, we are going to proceed along the same lines, except that 
strategies for Una will be based on a countable state automaton rather 
than a finite one. The states of the automaton are labeled by Una's belief
that the system is in state 1 \emph{under the assumption that Ian is playing
$\sigma^*$}.
In this section, we identify strategies for Una that are candidates for 
this purpose. The proof that they have the correct property (that $\sigma^*$ 
is a best response to the strategies $\tau_p$ that we construct) is in 
the next two sections.

As follows from Lemma \ref{lem:payoff}, any strategy of Una is a
best response to $\sigma^*$. 

In the case $\frac 12\le p\le \frac 23$, one can check that the range of 
$f_0$ is in $[1-p,\frac12]$, while the range of $f_1$ is in $[\frac12,p]$. 
Thus if Ian is playing $\sigma^*$, his last move
is sufficient to determine whether Una believes that it is 
more likely that the system
is in state $S_1$ or $S_0$. The strategy $\tau^*$ proposed for Una 
is a mixed strategy,
playing 1 with probability $(2p-1)/(4p-1)$ and 0 with probability $2p/(4p-1)$ if 
$\theta>\frac12$ and with the reverse probabilities otherwise (see Figure
\ref{fig:2state}).  In \cite{HRSV}, it is proved that $\sigma^*$ is a
best response to $\tau^*$ hence $(\sigma^*,\tau^*)$ is a Nash equilibrium. 

In the case $p=p^*$, if Ian is playing $\sigma^*$,
it turns out there are only 4 possible values attained by 
Una's belief that the system is in state $S_1$.
Namely, we have $1-p < f_1(1-p) < f_0(p) < p$ and $f_1$ maps 
$1-p$, $f_1(1-p)$, $f_0(p)$ and $p$ to $f_1(1-p)$, $f_0(p)$, $p$ and $p$ respectively. 
Similarly $f_0$ maps $1-p$, $f_1(1-p)$, $f_0(p)$ and $p$ to $1-p$, $1-p$, 
$f_1(1-p)$ and
$f_0(p)$ respectively. \cite{HRSV} shows that $\sigma^*$ is a best response to 
a strategy $\tau^{**}$ (and hence
an equilibrium strategy), given by a four state automaton corresponding 
to these four values
of $\theta$ together with rules corresponding to the above: if Ian plays 1, then
the automaton moves one step to the right; if Ian plays 0, then the automaton
moves one step to the left (see Figure \ref{fig:fourstate}). 
In each state of the automaton, there is an associated probability 
distribution on Una's choice of 0 or 1, which they exhibit explicitly.

\begin{figure}\includegraphics[width=1.7in]{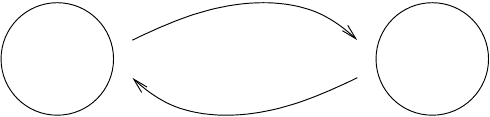}
\caption{For $p<\frac 23$, Una's automaton has two states, capturing whether
she believes it's more likely the system is in $S_1$ or $S_0$.
Whether $\theta>\frac 12$ or $\theta<\frac 12$ (but not the 
actual value of the belief) depends solely on Ian's last move.}\label{fig:2state}
\end{figure}

\begin{figure}
\includegraphics[width=4.5in]{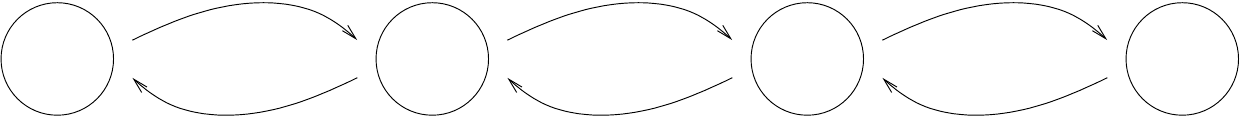}
\caption{For $p=p^*$, there are exactly 4 values of the Una's belief that may be 
attained starting from $\theta=\frac 12$. Una's automaton has 4 states,
one for each value of the Una's belief. 
Transitions between states are completely determined by Ian's moves.}
\label{fig:fourstate}
\end{figure}

Our results are based on exhibiting strategies for Una for which she
plays 0 and 1 with non-zero probabilities that depend solely on her belief that
the system is in state $S_1$ (assuming that Ian is playing $\sigma^*$).
Since Una's beliefs evolve in a manner that 
only depends on Ian's actions, we may once again describe her strategy by
an automaton. The principal differences are: (1) the automaton generally
has a countable number of states; and (2) the entire structure
of the automaton depends on $p$. An example of such an automaton 
is shown in Figure \ref{fig:infstate}.

\begin{figure}
\includegraphics[width=4.5in]{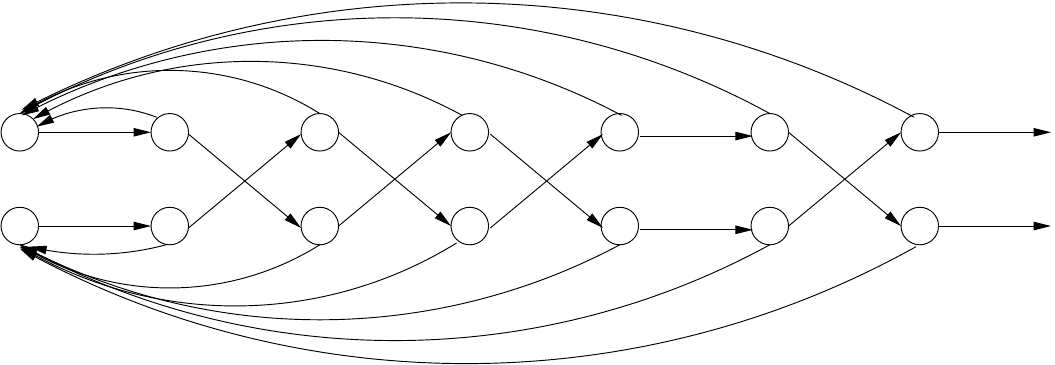}
\caption{Una's automaton for $p=0.72$. The states on the left 
of the diagram are those where the belief of Una is $p$ or $1-p$. Each state
corresponds to a value of $\theta$. Those in
the upper half of the diagram are those where Una believes it is more likely 
the system is in state $S_1$. If a state $\theta$ is in the upper half, 
its mirror image in the lower half is $1-\theta$. 
For states in the upper half of the diagram, if
Ian plays 1, the state returns to $p_0=p$, while if Ian plays 0,
the state advances to the right. In the lower half of the diagram,
if Ian plays 0, the state returns to $1-p$, while it advances if Ian plays 1.
The pattern of which arrows switch sides and which continue depends on $p$.
}\label{fig:infstate}
\end{figure}

The pattern of arrows is completely determined by $p$. The description of the 
strategy will be complete once we specify for each state, the probability of playing
1. Recall that the states are labelled by $(p_n)_{n\ge 0}$ and $(1-p_n)_{n\ge 0}$. 
If the automaton is in state $\theta$, we will define $x(\theta)$ to be the probability
that Una chooses 1. In this case, we will say that $(x(\theta))_{\theta\in[1-p,p]}$
is the strategy that Una is playing.

As mentioned above, to show that $\sigma^*$ is optimal, it suffices 
to find a strategy $x(\theta)$, to which $\sigma^*$ is the
best response. We therefore suppose that a particular strategy 
$x(\theta)$ has been selected
by Una, and we ask whether $\sigma^*$ is a best response for Ian. 
We will show that for certain $p$, we can exhibit an $x(\theta)$, solving 
the equations \eqref{eq:MDP} for Ian.

The state space that we use for Ian will consist of a pair $(\theta,s)$, where
$\theta$ is Una's belief that the system is in state $S_1$
and $s\in\{S_0,S_1\}$ is the state of the system. 
We define $(x(\theta))_{\theta\in[0,1]}$ recursively and give 
sufficient conditions for it to define a strategy for Una to which $\sigma^*$
is a best response. For the time being, we restrict attention
to the case $\Phi^n(p)\ne \frac 12$ for all $n$. This excludes 
countably many values of $p$. We set $\gamma=2p-1$ since this
is a quantity that occurs frequently.

Let $A_0=\begin{pmatrix}\gamma&-\gamma\\1-p&p\end{pmatrix}$,
$A_1=\begin{pmatrix}p&1-p\\-\gamma&\gamma\end{pmatrix}$,
$b_0=\begin{pmatrix}1\\0\end{pmatrix}$ and
$b_1=\begin{pmatrix}0\\1\end{pmatrix}$.
Let $\epsilon(\theta)=1$ if $\theta>\frac 12$ and 0 if $\theta<\frac12$. 
For $\theta\in[0,1]$, let $\eta_n(\theta)=\epsilon(\Phi^n(\theta))$.

Let $\iota_n=\eta_n(p)$.
Define $\vec w$ by
$$
\vec w=(I+A_{\iota_0}+A_{\iota_0}A_{\iota_1}+\ldots)\begin{pmatrix}
1\\1\end{pmatrix}.
$$

Define quantities $v$ and $Z$ (both depending on $p$) by 
\begin{equation}\label{eq:prevZ}
\begin{split}
v&=1/(pw_1+(1-p)w_2)\\
Z&=(w_1-w_2)v/2.
\end{split}
\end{equation}

\begin{prop}\label{prop:suffcond1}
Let $\frac 12<p < \frac 12+\frac{\sqrt{3}}{6}\approx 0.789$. 
Suppose that $\Phi^n(p)\ne \frac 12$
for each $n$ and let $v$ and $Z$ be as above. 
There is a unique solution to the equations
\begin{equation}
\label{eq:GHeqs}
\begin{split}
\begin{pmatrix}V_1(\theta)\\V_0(\theta)\end{pmatrix}&=
A_{\epsilon(\theta)} \begin{pmatrix}V_1(\Phi(\theta))\\ V_0(\Phi(\theta))\end{pmatrix}
-v\begin{pmatrix}1\\1\end{pmatrix} + (1-\gamma Z)b_{\epsilon(\theta)}
\text{ for $\theta\ne\tfrac 12$;}\\
V_1(\tfrac12)&=V_0(\tfrac12)=\tfrac 12-v-\gamma Z.
\end{split}
\end{equation}

Define $x(\theta)$ by
\begin{equation}\label{eq:prexdef}
x(\theta)=\begin{cases}
V_1(\theta)+v+\gamma Z&\text{if $\theta>\frac12$;}\\
1-(V_0(\theta)+v+\gamma Z)&\text{if $\theta< \frac12$;}\\
\tfrac12&\text{if $\theta=\frac12$,}
\end{cases}
\end{equation}

Suppose that the following inequalities are satisfied. 
\begin{equation}
\begin{split}\label{eq:preineqs}
&V_1(\theta)\ge  \gamma Z-v\text{ for $\theta<\tfrac12$}\\
&4\gamma Z\le 1,\\
-\gamma Z-v\le {}&{} V_1(\theta)\le 1-\gamma Z-v\text{ for $\theta >  \tfrac 12$.}
\end{split}
\end{equation}
Then $0\le x(\theta)\le 1$ for all $\theta$.
If $\tau$ is the strategy where Una plays 1 with probability
$x(\theta)$ if her belief that the system is in state $S_1$ is $\theta$,
then $\sigma^*$ is a best response to $\tau$
and the value of the game is $v$.
\end{prop}

\begin{proof}
One can check that for $p<\frac 12+\frac{\sqrt 3}6$
that the matrices $A_0$ and $A_1$ are strict contractions (with respect to
the Euclidean norm). 
Define the Banach space, $B=B([0,1],\R^2)$, of bounded 
$\R^2$-valued functions
on $[0,1]$ with norm given by $\|X\|=\sup_{\theta\in[0,1]}|X(\theta)|$.

We then define an operator, $\Lax$, on $B$ by
\begin{equation}\label{eq:preLax}
\Lax X(\theta)=
\begin{cases}
A_{\epsilon(\theta)}X(\Phi(\theta))-v\begin{pmatrix}1\\1\end{pmatrix}
+(1-\gamma Z)b_{\epsilon(\theta)}&\text{ if $\theta\ne\frac 12$;}\\
\begin{pmatrix}\frac 12-v-\gamma Z\\\frac 12-v-\gamma Z\end{pmatrix}
&\text{ if $\theta=\frac 12$.}
\end{cases}
\end{equation}
One sees that $\Lax$ is a contraction of $B$, and therefore has a 
unique fixed point, $X^*(\theta)=\left(\begin{smallmatrix}
V_1(\theta)\\V_0(\theta)\end{smallmatrix}\right)$. This establishes the 
first claim.

We now show that $V_1(p)=V_0(1-p)=-Z$ and $V_0(p)=V_1(1-p)=Z$. 
Since one has $\Phi(1-x)=1-\Phi(x)$ one sees that if $\left(\begin{smallmatrix}
V_1(\theta)\\V_0(\theta)\end{smallmatrix}\right)$ is a solution to \eqref{eq:GHeqs},
then so is $\left(\begin{smallmatrix}
V_0(1-\theta)\\V_1(1-\theta)\end{smallmatrix}\right)$. Hence, by uniqueness,
$V_1(\theta)=V_0(1-\theta)$. It follows that $x(1-\theta)=1-x(\theta)$. 

By iterating \eqref{eq:GHeqs} and using the fact that the $A_\epsilon$
are contracting, one obtains
\begin{align*}
\begin{pmatrix}V_1(p)\\V_0(p)\end{pmatrix}&=
-v(I+A_{\iota_0}+A_{\iota_0}A_{\iota_1}+\ldots)\begin{pmatrix}1\\1\end{pmatrix}\\
&+(1-\gamma Z)\left(b_{\iota_0}+A_{\iota_0}b_{\iota_1}+
A_{\iota_0}A_{\iota_1}b_{\iota_2}+\ldots\right)
\end{align*}
If one defines $\psi_i(x)=b_i+A_ix$, then the term in the last parentheses 
is $\lim_{n\to\infty}\psi_{\iota_0}\psi_{\iota_1}\ldots\psi_{\iota_n}
(\begin{smallmatrix}0\\0\end{smallmatrix})$. Since the $\psi_i$ are contracting
and have a common fixed point of $(\begin{smallmatrix}1\\1\end{smallmatrix})$,
we deduce this term is exactly this fixed point. 
Hence we have
\begin{equation*}
\begin{pmatrix}V_1(p)\\V_0(p)\end{pmatrix}=-v\vec w+(1-\gamma Z)
\begin{pmatrix}1\\1\end{pmatrix},
\end{equation*}
so that $V_1(p)=-vw_1+(1-\gamma Z)=-Z$ and $V_0(p)=-vw_2+(1-\gamma Z)=Z$ and then
$V_0(1-p)$ and $V_1(1-p)$ are $-Z$ and $Z$ respectively by the symmetry.

Now define $x(\theta)$ using \eqref{eq:prexdef} and assume the inequalities
\eqref{eq:preineqs} are satisfied. The final pair of inequalities of 
\eqref{eq:preineqs} ensures that $0\le x(\theta)\le 1$ for each $\theta>\frac 12$.
By the symmetry, one obtains $0\le x(\theta)\le 1$ for each $\theta$ as required.

Let $\tau$ be the strategy for
Una where if her belief is $\theta$, she plays 1 with probability $x(\theta)$. 
Then define $V(s,\theta)$ to be $V_1(\theta)$ if $s=S_1$ and $V_0(\theta)$ if $s=S_0$. 
We show that $\sigma^*$ is a best response to $\tau$ with average long-term
gain $v$. 

For \eqref{eq:MDP} to be satisfied, if $\theta>\frac 12$ and the system is
in state $S_1$, Ian should receive equal long-term gain from playing either
move (as he makes both with positive probability) whereas in state
$S_0$, he should make a larger gain by playing 0. In other words, to satisfy
\eqref{eq:MDP} if $\theta>\frac 12$, we require:
\begin{align*}
V_1(\theta)+v&=x(\theta)+pV_1(p)+(1-p)V_0(p)\\
&=pV_1(f_0(\theta))+(1-p)V_0(f_0(\theta))\\
V_0(\theta)+v&=1-x(\theta)+(1-p)V_1(f_0(\theta))+pV_0(f_0(\theta))\\
&\ge (1-p)V_1(p)+pV_0(p),
\end{align*}
with similar requirements when $\theta<\frac 12$.

Substituting the values for $V_1$ and $V_0$ at $p$ and $1-p$, these requirements 
are for $\theta>\frac 12$:
\begin{equation}\label{eq:reqs}
\begin{split}
V_1(\theta)+v&=x(\theta)-\gamma Z\\
&=pV_1(f_0(\theta))+(1-p)V_0(f_0(\theta))\\
V_0(\theta)+v&=1-x(\theta)+(1-p)V_1(f_0(\theta))+pV_0(f_0(\theta))\\
&\ge \gamma Z,
\end{split}
\end{equation}
again with similar requirements when $\theta<\frac 12$. 

The first equality of \eqref{eq:reqs} is satisfied by definition of $x(\theta)$
and the second is the first component of \eqref{eq:GHeqs}.
For the third equality, notice that by using the first two equalities
one has $1-x(\theta)=1-pV_1(\Phi(\theta))-(1-p)V_0(\Phi(\theta))-\gamma Z$. 
Now, the second component of \eqref{eq:GHeqs} gives
$V_0(\theta)+v=(1-2p)V_1(\Phi(\theta))+(2p-1)V_0(\Phi(\theta))+1-\gamma Z$.
Combining these, we obtain the third equality of \eqref{eq:reqs}. 
Finally the hypothesis that $V_1(\theta)\ge -\gamma Z-v$ together with the 
symmetry yields $V_0(\theta)+v\ge-\gamma Z$ giving the required inequality
in \eqref{eq:reqs}.
\end{proof}

\begin{figure}
\includegraphics[width=6cm]{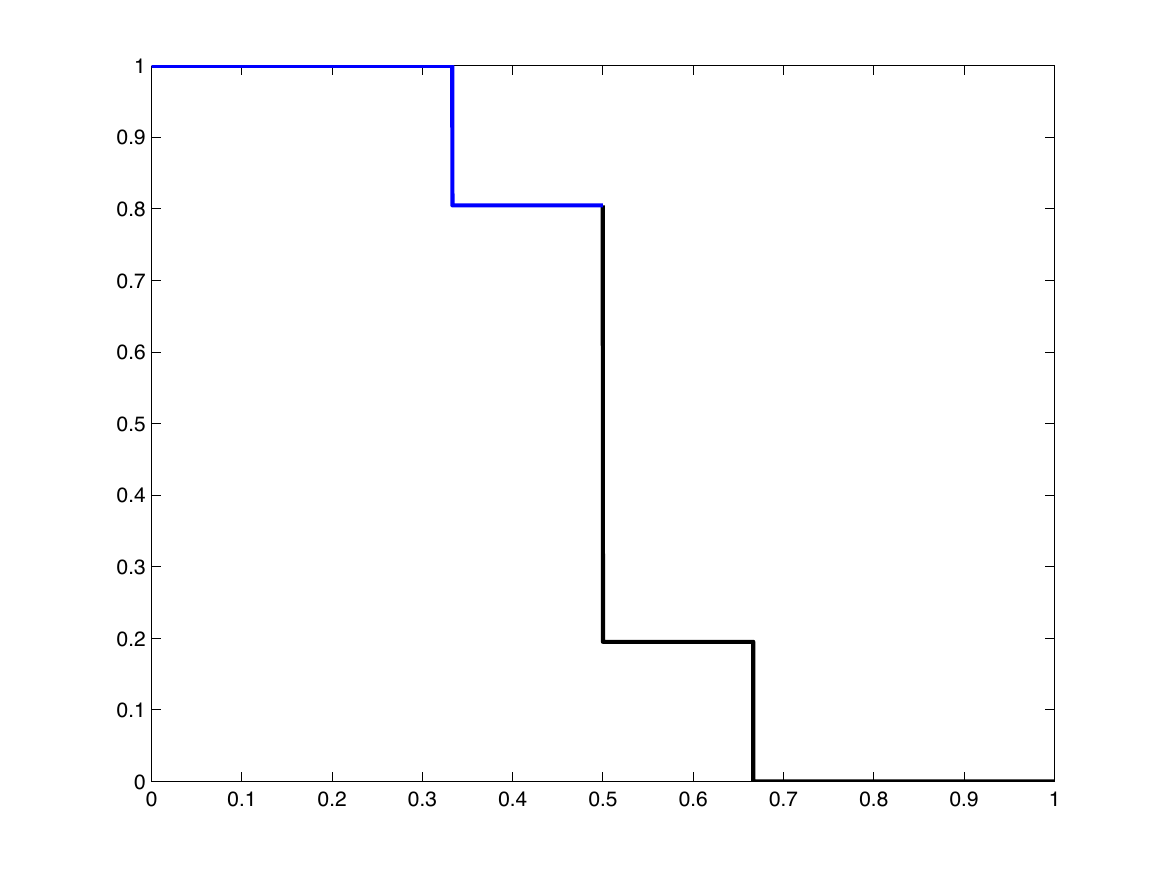}
\includegraphics[width=6cm]{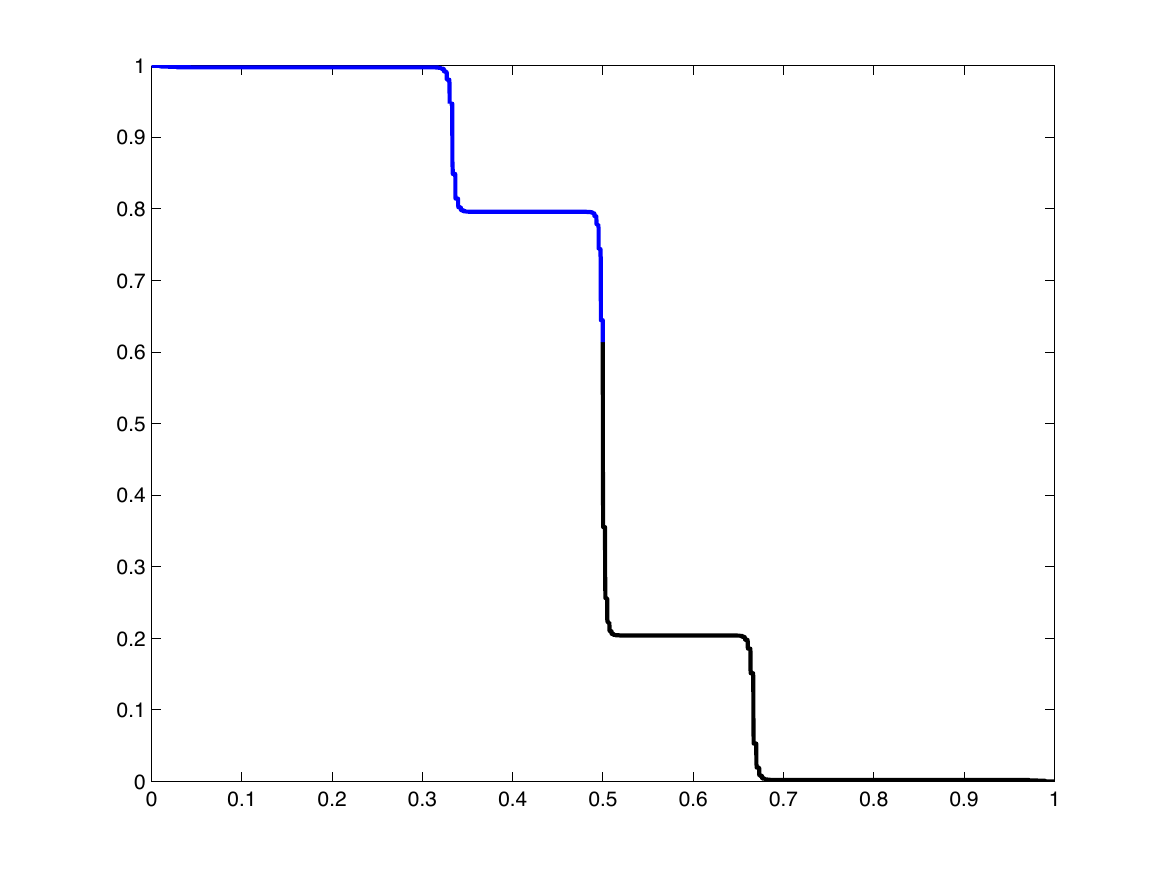}
\includegraphics[width=6cm]{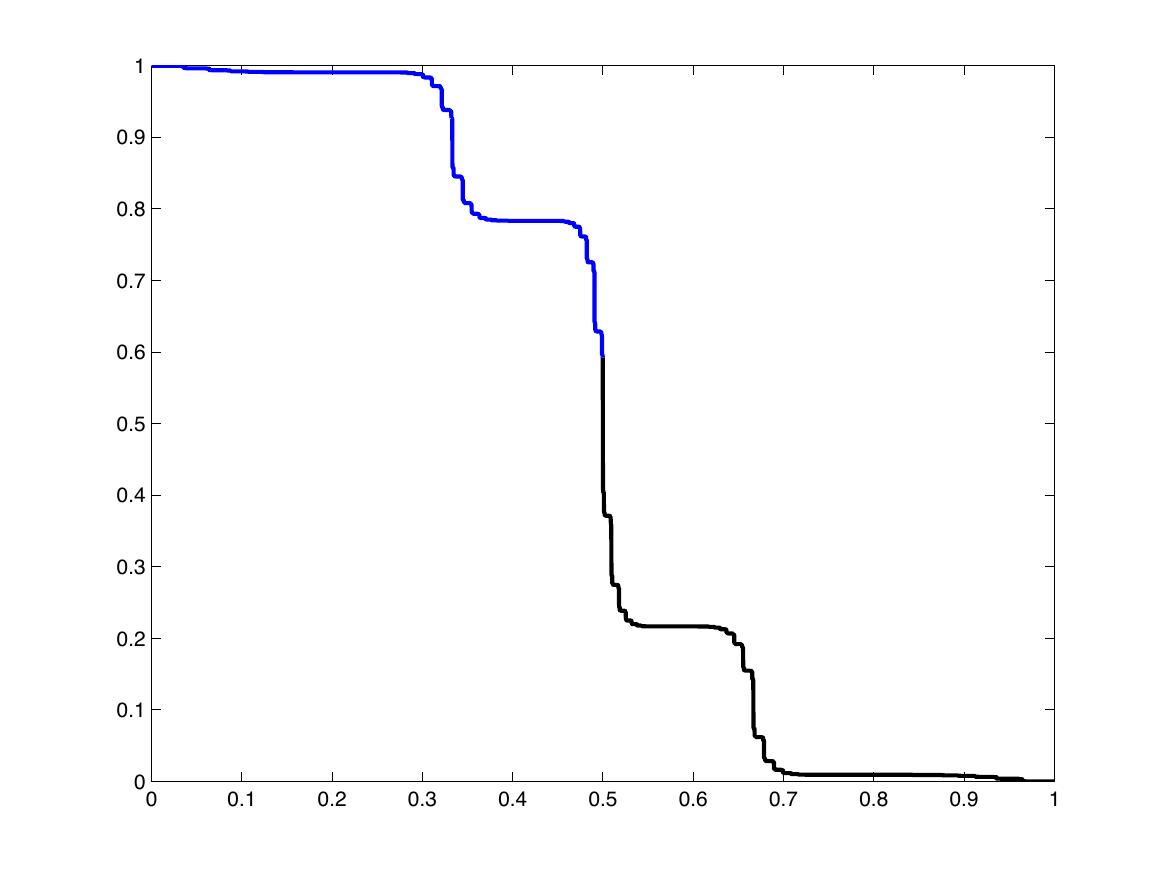}
\includegraphics[width=6cm]{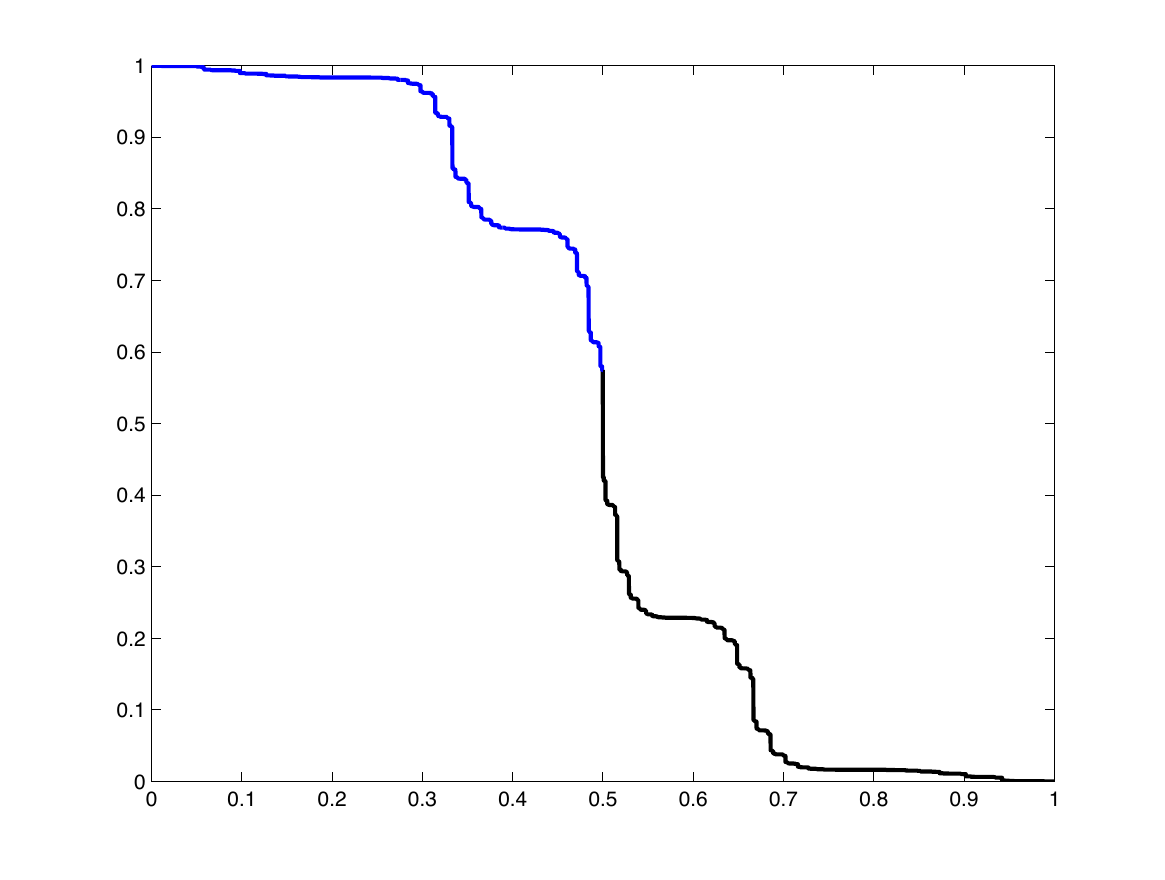}
\includegraphics[width=6cm]{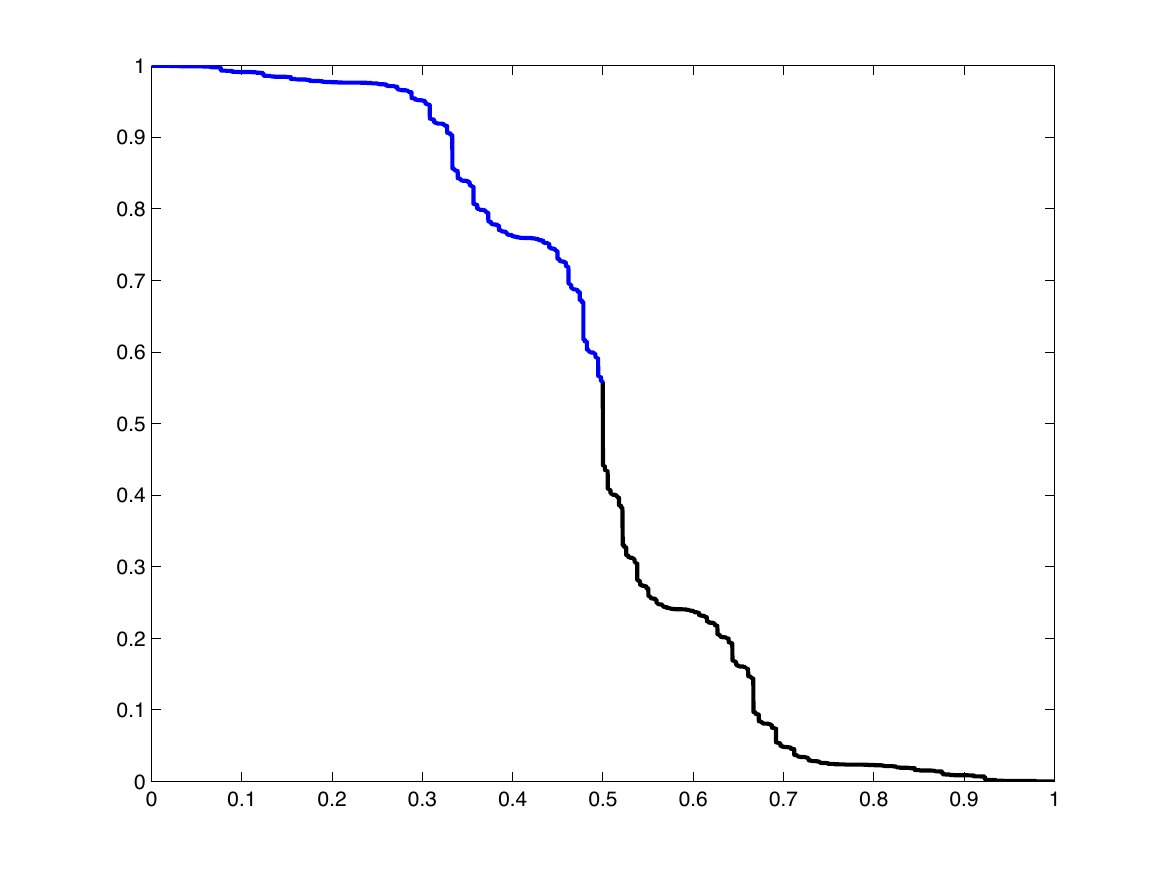}
\includegraphics[width=6cm]{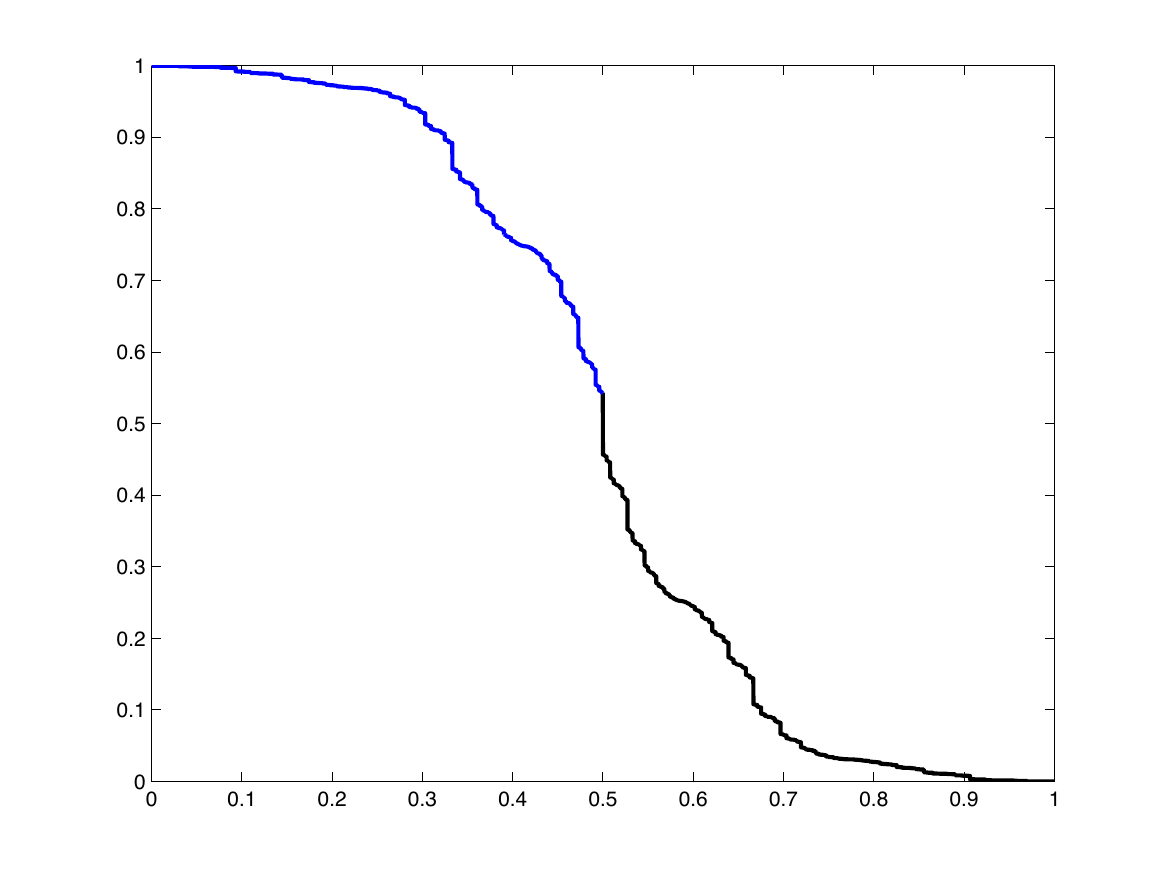}
\includegraphics[width=6cm]{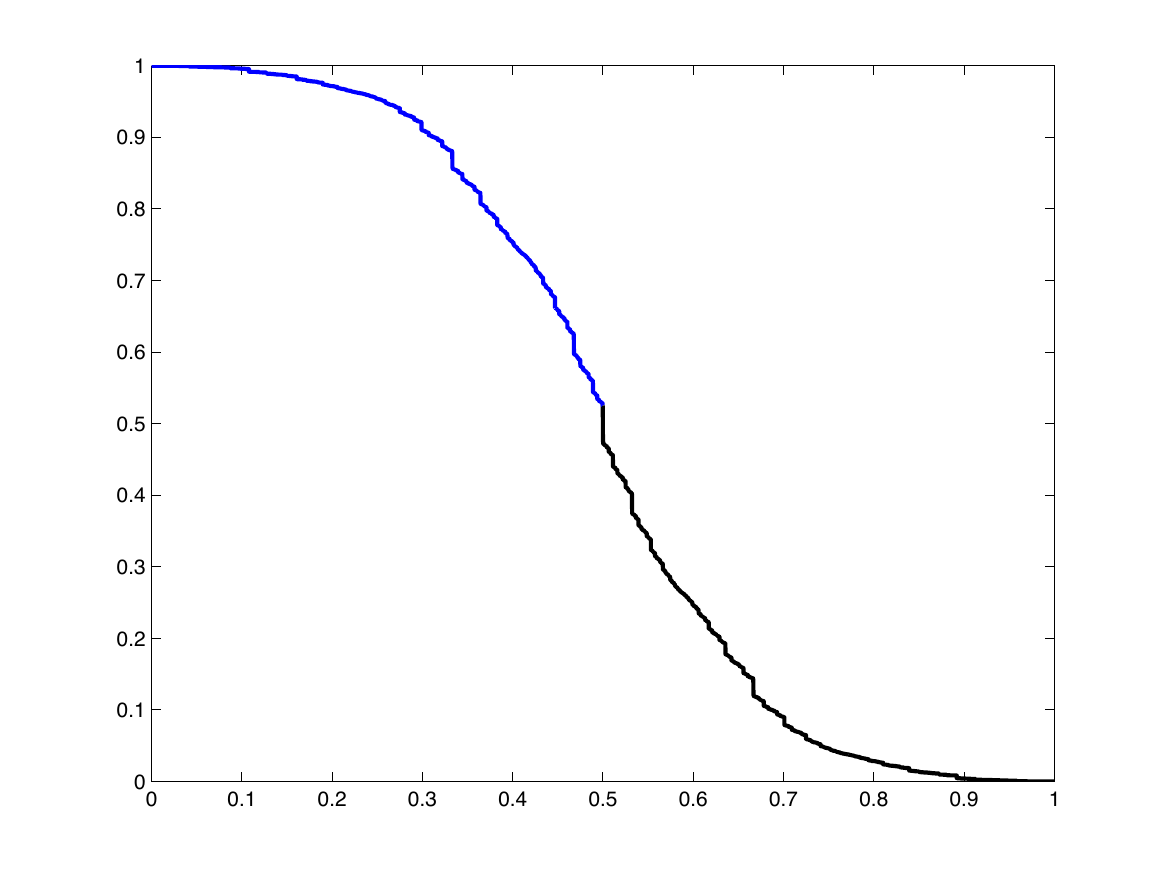}
\includegraphics[width=6cm]{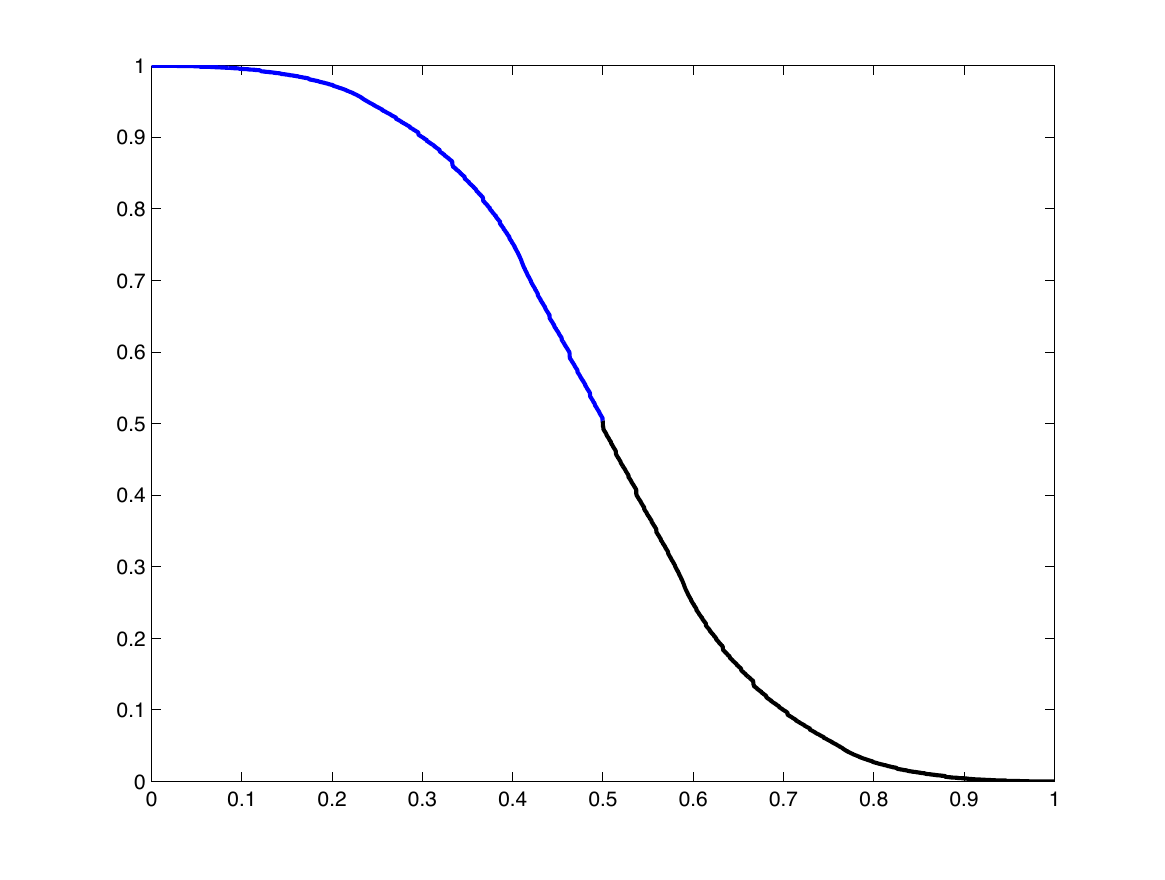}
\caption{The graphs of $\theta \mapsto x(\theta)$ for values of $p$ 
ranging from $p=0.6625$ to $p=0.7325$ in  steps of $0.01$.} 
\end{figure}

Notice that by \eqref{eq:prevZ} and Proposition \ref{prop:suffcond1},
we now have a second, apparently independent equation for the long-term
average gain, $v$, (whenever 
the conditions of the proposition are satisfied).
We verify that the expressions are equal as this reveals useful identities.

Starting from this second expression, we have
\begin{align*}
1/v&=\begin{pmatrix}p&1-p\end{pmatrix}
(I+A_{\epsilon_0}+A_{\epsilon_0}A_{\epsilon_1}+
A_{\epsilon_0}A_{\epsilon_1}A_{\epsilon_2}
+\ldots)\begin{pmatrix}1\\1\end{pmatrix}\\
&=\begin{pmatrix}1&1\end{pmatrix}(I+A^T_{\epsilon_0}+A^T_{\epsilon_1}
A^T_{\epsilon_0}+A^T_{\epsilon_2}A^T_{\epsilon_1}A^T_{\epsilon_0}
+\ldots)\begin{pmatrix}p\\1-p\end{pmatrix}
\end{align*}

Notice that
$$
\begin{pmatrix}1&0\\1&1\end{pmatrix}A_\epsilon^T
\begin{pmatrix}1&0\\-1&1\end{pmatrix}=U_\epsilon,
$$
for $\epsilon\in\{0,1\}$. 

Accordingly, we can rewrite the expression for $1/v$ as
\begin{align*}
&\begin{pmatrix}1&1\end{pmatrix}
\begin{pmatrix}1&0\\-1&1\end{pmatrix}
(I+U_{\epsilon_0}+U_{\epsilon_1}U_{\epsilon_0}+U_{\epsilon_2}U_{\epsilon_1}U_{\epsilon_0}
+\ldots)
\begin{pmatrix}1&0\\1&1\end{pmatrix}
\begin{pmatrix}p\\1-p\end{pmatrix}\\
=&\begin{pmatrix}0&1\end{pmatrix}
(I+U_{\epsilon_0}+U_{\epsilon_1}U_{\epsilon_0}+
U_{\epsilon_2}U_{\epsilon_1}U_{\epsilon_0}+\ldots)
\begin{pmatrix}p\\1\end{pmatrix}.
\end{align*}
This expression matches the one that we found in \eqref{eq:vval1}.

\section{Conditions for monotonicity}
\label{sec:monotonicity}
To prove that the inequalities \eqref{eq:preineqs} are satisfied 
(in a range of values of 
$p$) we are going to show that $V_1$ (and $V_0$) are monotonic, and we 
control the boundary values.
For convenience, we work in this section with a dynamical system 
$\alpha$ derived from $\Phi$, namely $\alpha(t)=\max(\Phi(t),1-\Phi(t))$.
This exploits the symmetry of $\Phi$ (that $\Phi(1-t)=1-\Phi(t)$) and chooses
the representative of each $t$ in the interval $[\frac 12,1]$.
In this section, we show that the monotonicity
conditions follow from a \emph{pressure} condition for a 
given potential for the dynamical system $\alpha$. We write
$\alpha^n(t)$ for the $n$-fold iterate of the map $\alpha$. It is easy
to verify that $\alpha^n(t)=\max(\Phi^n(t),1-\Phi^n(t))$, where $\Phi^n$
similarly denotes the $n$-fold iterate of $\Phi$.

Given a parameter $p$ and a function $g$ on $[\frac 12,1]$, we define 
the $\alpha$-\emph{pressure} of $g$ to be
\begin{equation*}
P_{\alpha_p}(g):=\limsup_{n\to\infty}\frac 1n\log\sum_{x\in\alpha^{-n}(\frac 12)}
\exp\left(\sum_{i=0}^{n-1}g(\alpha^i x)\right).
\end{equation*}
Pressure, introduced by Ruelle in \cite{Ruelle}, is a dynamical analogue of
the partition function in statistical mechanics. Its value is a combination
of the long-term average value along orbits of $g$ with the complexity of 
$\alpha_p$. The definition here is not equivalent to Ruelle's. Notice that 
$P_{\alpha_p}(0)$ simply counts the growth rate of the number of 
pre-images of $\frac 12$. 
Since $\alpha^n$ is a piecewise monotonic function where the direction of 
monotonicity changes at $x$ exactly when $\alpha^j(x)=\frac 12$ for some $j\le n$,
$P_{\alpha_p}(0)$ is precisely the logarithmic growth rate of the number
of intervals of monotonicity of $\alpha^n$. This quantity was shown by 
Misiurewicz and Szlenk \cite{MisSzl} and by Young \cite{Young} to be equal to
the \emph{topological entropy} of $\alpha_p$ (which is also equal to 
Ruelle's pressure evaluated at $g\equiv 0$). 
Topological entropy is a standard measure
of complexity for a continuous dynamical system. 

We will need the following simple result independent of our specific context: 
\begin{lem}\label{lem:purejump}
Let $(a_k)$ be a sequence in $[0,1]$ with $a_k\ne a_{k'}$ for all $k\ne k'$
and let $(b_k)$ be a summable sequence of non-negative numbers. 
Suppose that $(f_n)$ is a sequence of real-valued functions defined on 
$[0,1]$, each of pure
jump type. Suppose further that the only discontinuities of $(f_n)$ 
occur at the $a_k$'s 
and that $|\Delta f_n(a_k)|\le b_k$ for each $k$ and $n$,
where $\Delta f(x)=\lim_{t\downarrow x}f(t)-\lim_{t\uparrow x}f(t)$. 
If $\|f_n-f\|_\infty\to 0$,
then $f$ is of pure jump type with discontinuities only at the $a_k$'s. 
The magnitude of the discontinuity of $f$ at $a_k$ is bounded above by $b_k$.
\end{lem}

\begin{proof}
Denote $V f$ the total variation of $f$:
$$
Vf=\sup_m\sup_{0=x_0\le x_1\le\ldots\le x_m=1}\sum_{i=1}^m|f(x_i)-f(x_{i-1})|.
$$
Let $V_I f$ be the variation of $f$ 
on the interval $I$. 

For any $0=x_0\le\ldots\le x_m=1$, notice that $\sum_{i=1}^m|f_n(x_i)-f_n(x_{i-1})|
\to \sum_{i=1}^m|f(x_i)-f(x_{i-1})|$. Hence since the left side is uniformly bounded
by $\sum_k b_k$ for all $n$ and for all $0=x_0\le\ldots\le x_m=1$, we deduce that
$f$ has bounded variation.

Hence it has a unique (up to additive constants) Lebesgue decomposition
as a sum $f_c+f_d$ where $f_c$ is continuous and $f_d$ has only jump-type 
discontinuities. It is known that $Vf=Vf_c+Vf_d$. 
For any $\epsilon>0$, there exists a $K$ such that $\sum_{k\ge K}b_k<\epsilon$.
Letting $I_1,\ldots,I_M$ be any disjoint collection of intervals avoiding
the $a_k$'s with $k<K$, we see that $\sum_{i=1}^M V_{I_i}f<\epsilon$.
In particular, we deduce $Vf_c<\epsilon$ for arbitrary $\epsilon$ so that $f$ has 
pure jump type. We also deduce that $f_d$ cannot have any jump discontinuities 
other than at the $a_k$'s and the result is proven.
\end{proof}

\begin{prop}\label{prop:pressurecond}
Let $\frac 12 < p< \frac 12+\frac{\sqrt 3}6$
be such that $P_{\alpha_p}(\log h)<0$ (where $h(t)=\gamma/t$ and
$\gamma$, as before, is defined to be $2p-1$). Then the conditions 
\eqref{eq:preineqs} of Proposition \ref{prop:suffcond1} are satisfied.
Hence $\sigma^*$ is an optimal strategy for Ian for the game with this value of $p$.
\end{prop}

\begin{proof}
First, assume that $p$ is such that $\Phi^n(p)\ne \frac 12$ for all $n$ as
this is a hypothesis for Proposition \ref{prop:suffcond1}.
Let $X^0(\theta)=\begin{pmatrix}-\gamma Z\\-\gamma Z\end{pmatrix}$
and set $X^n=\Lax^n(X_0)$. Since $\mathcal L$ is a contraction mapping,
we have $\left\|X^n-X^*\right\|\to 0$,
where $X^*(\theta)=\begin{pmatrix}V_1(\theta)\\V_0(\theta)\end{pmatrix}$
is the fixed point of $\Lax$
from Proposition \ref{prop:suffcond1}. Notice that
$\Lax$ preserves the set of functions $\left\{\begin{pmatrix}f_1(\theta)\\f_2(\theta)
\end{pmatrix}\colon f_2(\theta)=f_1(1-\theta)\right\}$.

From the contraction mapping
theorem, there exists $M>0$ such that $|\Delta X^n(\frac 12)|\le M$ 
for all $n$. From \eqref{eq:preLax}, we observe that for $\theta\ne\frac 12$,
\begin{equation}\label{eq:disconts}
\Delta X^n(\theta)=A_{\epsilon(\theta)}\Delta X^{n-1}(\Phi(\theta)).
\end{equation}

Notice that $X^n$ only has discontinuities at pre-images of $\frac 12$ of order 
at most $n$ and is piecewise constant between discontinuities. We now show
that if $P_{\alpha_p}(\log h)<0$, then the conditions of Lemma
\ref{lem:purejump} are satisfied by the components of $X^n(\theta)$
and that $V_1(\theta)$ and $V_0(\theta)$
are monotonically decreasing and increasing respectively. 

Suppose that $\theta>\frac 12$. Then we have $\Phi(\theta)=f_0(\theta)=
(3p-1)-\gamma/\theta$
and we calculate
$$
A_1\begin{pmatrix}\Phi(\theta)-1\\ \Phi(\theta)\end{pmatrix}
=\frac{\gamma}{\theta}\begin{pmatrix}\theta-1\\ \theta\end{pmatrix}.
$$
Similarly, if $\theta<\frac 12$, we have
$$
A_0\begin{pmatrix}\Phi(\theta)-1\\ \Phi(\theta)\end{pmatrix}
=\frac{\gamma}{1-\theta}\begin{pmatrix}\theta-1\\ \theta\end{pmatrix}.
$$
If $t\in\Phi^{-n}(\frac 12)$, let $\theta_i=\Phi^{n-i}(t)$ (so that $\theta_0=\frac12$
and $\theta_n=t$) and
let $\epsilon_i=\epsilon(\theta_i)$ and $u=\max(t,1-t)$. The above shows
$$
A_{\epsilon_i}\begin{pmatrix}\theta_{i-1}-1\\ \theta_{i-1}\end{pmatrix}
=\frac{\gamma}{\max(\theta_i,1-\theta_i)}
\begin{pmatrix}\theta_{i}-1\\ \theta_{i}\end{pmatrix}.
$$
Combining these equalities gives
\begin{equation}\label{eq:coc}
A_{\epsilon_n}\cdots A_{\epsilon_1}\begin{pmatrix}-\frac12\\\tfrac12\end{pmatrix}
=\prod_{i=1}^{n}\frac{\gamma}{\max(\theta_i,1-\theta_i)}
\begin{pmatrix}\theta_n-1\\ \theta_n\end{pmatrix}.
\end{equation}

Using \eqref{eq:disconts}, we see that for $m>n$, one has
\begin{equation*}
\Delta X^m(t)=A_{\epsilon_n}\ldots A_{\epsilon_1}\Delta X_{m^n}(\tfrac 12)
\end{equation*}

By symmetry,
we see that $\Delta X^k(\frac12)$ is a
multiple of $\left(\begin{smallmatrix}-1\\1\end{smallmatrix}\right)$ for all $k$.
Since $1-p\le t\le p$, one obtains
$$
\big|\Delta X^m(t)\big|\le C
\prod_{i=1}^n \frac{\gamma}{\max(\theta_i,1-\theta_i)},
$$
where $C$ does not depend on $t$, $m$ or $n$.
This can be re-expressed in terms of $\alpha$ by 
$|\Delta X^m(t)|\le C\prod_{i=1}^n (\gamma/\alpha^i(u))$ for all $m$,
where $u=\max(t,1-t)$ satisfies $\alpha^n(u)=\frac 12$. Hence
we see the hypothesis, $P_{\alpha_p}(\log h)<0$, ensures
that the conditions for Lemma \ref{lem:purejump} are satisfied.
Hence $V_1$ and $V_0$ of pure jump type.

The jumps satisfy $\Delta X^*(\theta)=
A_{\epsilon(\theta)}\Delta X^*(\Phi(\theta))$. By \eqref{eq:coc}, they
are all of the same sign.
Now provided the pressure is negative and $p$ is not a pre-image of $\frac 12$,
we check using \eqref{eq:GHeqs} that $V_1(\frac12^+)=\gamma Z-v$,
$V_1(\frac 12^-)=1-3\gamma Z-v$ so that $\Delta V_1(\frac 12)=
4\gamma Z-1$.
On the other hand, the total of all discontinuities (all of the same sign) is
$-2Z$. In order for these to have the same sign, one sees that $Z>0$ and $4\gamma Z<1$.
The function $V_1(\theta)$ is therefore a decreasing function.

Now to check \eqref{eq:preineqs}, it suffices to show that
$V_1(\frac 12^-)\ge \gamma Z-v$, $V_1(\frac 12^+)\le 1-\gamma Z-v$
and $V_1(p)\ge -\gamma Z-v$. The first two of these follow from
the fact that $Z>0$ and $4\gamma Z<1$. To verify the last inequality,
we note that the above contraction argument works outside the range $[1-p,p]$,
so that $V_1$ is monotonic on all of $[0,1]$. Since $\Phi(1)=p$, we apply 
\eqref{eq:GHeqs} to see that $V_1(1)=-\gamma Z-v$, so that the third
inequality is satisfied by monotonicity.

In the case where $p$ is a preimage of $\frac 12$, the above expressions
for $V_1(\frac 12^+)$ and $V_1(\frac 12^-)$ are no longer valid as 
$V_1$ and $V_0$ are discontinuous at $1-p$ and $p$. The essential modification
is to show that $V_1(\frac 12^+)-V_1(\frac 12)=V_1(\frac 12)-V_1(\frac 12^-)$. 
The matrix equalities \eqref{eq:GHeqs} then ensure that at each 
$x\in\bigcup_n\Phi^{-n}(\frac 12)$, one has that $V_1(x)$ is the average 
of $V_1(x^-)$ and $V_1(x^+)$ (and similarly for $V_0$). 
For a fixed $p$, this allows us to deduce monotonicity and verify inequalities 
on entire intervals of $\theta$ values by checking the inequalities at a
finite collection of points as before.

\end{proof}

\section{Pressure bounds}
\label{sec:pressure}
In this section, we find ranges of $p$ where $P_{\alpha_p}(\log h)<0$ is satisfied
(so that $\sigma^*$ is an optimal strategy for Ian).
Indeed if $p\in[\frac 12,\frac 23]$, then $\alpha^{-n}(\frac 12)$ 
is empty, so that trivially $P_{\alpha_p}(\log h)<0$.

Henceforth, we assume $p>\frac 23$. 
Notice that $\Phi(t)<\frac 12$ if and only if $t<\frac 23$. 
The map $\alpha$ can therefore be expressed as:
$$
\alpha(t)=
\begin{cases}
2-3p+\gamma/t&\text{if $t<\frac 23$;}\\
3p-1-\gamma/t&\text{if $t\ge \frac23$.}
\end{cases}
$$
This map is unimodal: monotone decreasing on the left branch and
increasing on the right branch
with $\alpha([\frac 12,\frac 23])=\alpha([\frac 23,1])=[\frac 12,p]$. 
We write 
$h^{(n)}(t)=h(t)\cdot h(\alpha(t))\cdots h(\alpha^{n-1}(t))$.

We partition $[\frac 12,p]$ into sub-intervals, counting possible transitions
between pairs of intervals, and over-estimating $\psi$ on the intervals to give
a rigorous, finitely-calculable estimate for the pressure in various ranges of $p$. 

It turns out that for $p$ in the range $[\frac 23,p^*]$ (where 
$p^*\approx 0.7589$ is the special $p$-value identified by H\"orner, Rosenberg, Solan
and Vieille in \cite{HRSV}), the map 
$\alpha$ is \emph{renormalizable}. That is, there are disjoint intervals
$I_1$ and $I_2$ with $I_1$ containing the critical point such that $\alpha(I_1)
\subset I_2$ and $\alpha(I_2)\subset I_1$. Since $\alpha|_{I_2}$ is monotonic, 
we see that the \emph{renormalized map}, $\alpha^2\colon I_1\to I_1$,
is a unimodal map. 
If $\alpha$ is renormalizable, then $I_1\cup I_2$ is an absorbing set. 
Points outside $I_1\cup I_2$ either eventually land in $I_1\cup I_2$ under iteration
or converge to fixed points so that all of the `interesting dynamics' lies in 
$I_1\cup I_2$. When a map is renormalizable, it decreases 
the growth rate of the number of iterated preimages lying in $I_1\cup I_2$:
an element of $I_1$ has at most one preimage in $I_2$ and an element of $I_2$
has at most two preimages in $I_1$, so that for $x\in I_1\cup I_2$,
$|\alpha^{-n}(x)\cap (I_1\cup I_2)|\le 2^{\lceil n/2\rceil}$. 
The renormalization is illustrated for $p=0.73$ in Figure \ref{fig:renorm}.

\begin{figure}
\includegraphics[width=2.5in]{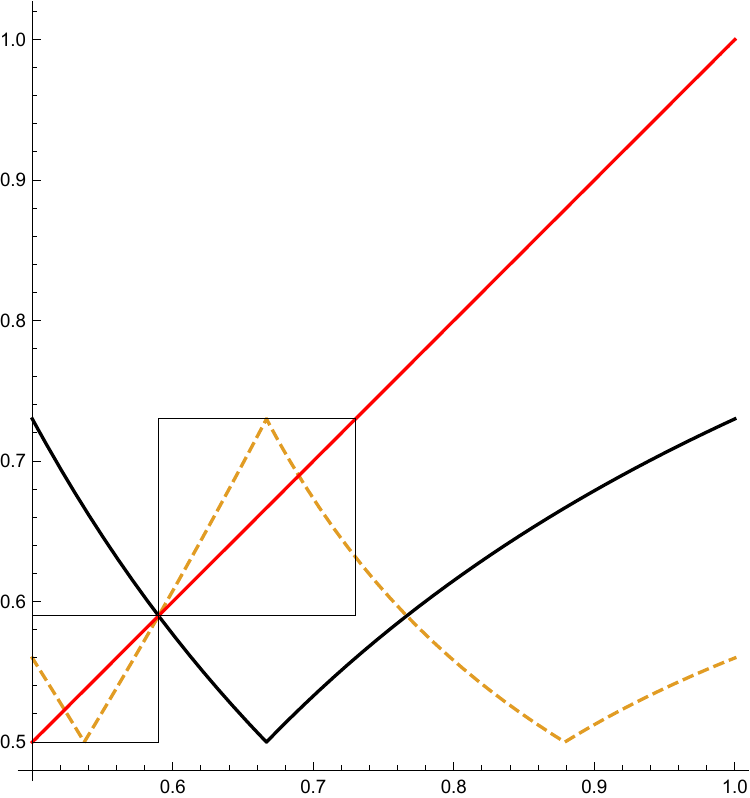}
\caption{The graph of $\alpha(x)$ (in black) for $p=0.73$. The graphs of 
$\alpha^2(x)$ (dashed) and $x$ are superimposed to illustrate the 
renormalization. The squares illustrate the fact that $\alpha^2$ maps 
$I_1$ and $I_2$ to themselves.
}
\label{fig:renorm}
\end{figure}

To see that $\alpha_p$ is renormalizable for $p\in [\frac 23,p^*]$, let
$p_i=\alpha^i(p)$ (so that $p_0=p$). One can check that for $p\in [\frac 23,p^*]$,
\begin{equation}\label{eq:renorm}
\begin{split}
&\tfrac 12<p_3<p_1<p_2<\tfrac 23<p;\\
&\alpha([\tfrac 12,p_1])=[p_2,p]\text{ and }
\alpha([p_2,p])=[\tfrac 12,p_1],
\end{split}
\end{equation}
establishing the (one-time) renormalizability. At the endpoint $p^*$ of the range
of $p$-values that we are considering, 
one has $p_2=p_1$ and hence $p_n=p_1$ for all $n\ge 1$.

It may happen that the renormalized map is itself renormalizable. 
This is the case for $p<0.709637$ and is illustrated in 
Figure \ref{fig:renorm4}. See Devaney's book \cite{Devaney}
for more information about renormalization of unimodal maps
and the relationship between interval maps and symbolic dynamics.

\begin{prop}\label{prop:pressureest}
Let $\alpha$ be a continuous piecewise monotonic map of $I=[\frac 12,1]$
and let $g$ be a continuous on $I$. 
Suppose that $I$ is partitioned into intervals $J_0,\ldots,J_{k-1}$.
Let $\beta_i=\max_{x\in J_i}h(x)$. Let the \emph{multiplicity} 
$m_{ij}=\max_{y\in J_j}\#
\{x\in J_i\colon \alpha(x)=y\}$. Let $A$ be the $k\times k$ matrix with entries
$a_{ij}=\beta_im_{ij}$. Then
$$
P_{\alpha}(\log h)\le \log\rho(A),
$$
where $\rho(A)$ denotes the spectral radius (i.e. maximal eigenvalue) of $A$. 
\end{prop}

\begin{proof}
Notice that there are at most $m_{i_0i_1}m_{i_1i_2}\ldots
m_{i_{n-1}i_n}$ $n$th order preimages $x$ of a point $y$ in $J_{i_n}$ with the property
that $\alpha^t(x)\in J_{i_t}$ for each $0\le t<n$. For each such preimage,
the largest possible contribution to the sum is 
$\beta_{i_0}\ldots\beta_{i_{n-1}}$, so that we see 
$$
\sum_{t\in\alpha_p^{-n}(\frac12)}h^{(n)}(t)
\le \sum_{i}(A^n)_{ij},
$$
where $j$ is the index of the interval containing $\frac 12$.
Taking logarithms and dividing by $n$, 
the result follows.
\end{proof}

For a fixed $p$ and any partition of $[\frac 12,1]$ into intervals, 
one can calculate the matrix $A$ so that this
proposition gives an upper bound for $P_{\alpha_p}(\log h)$. 
Hence in order to establish
that $P_{\alpha_p}(\log h)<0$, it suffices to exhibit \emph{some} finite partition 
such that the corresponding matrix $A$ has spectral radius less than 1.

In fact, when dealing with $\alpha_p$, the interval $[p,1]$ plays no role in the
pressure computation as points in this interval have no preimages. It therefore
suffices to partition the interval $[\frac 12,p]$.
A natural choice of intervals $J_0,\ldots,J_{k-1}$ is
obtained by taking the points $\frac 12$ and $(\alpha^i(p))_{i=0}^{k-1}$
in increasing order as the endpoints of intervals. The reason this choice is a good
one is that the endpoints of each of these intervals (except $\alpha^{k-1}(p)$)
are mapped exactly into each other, so that for most pairs $i$ and $j$, each
point in $J_j$ has \emph{exactly}
$m_{ij}$ preimages in $J_i$, making the estimates reasonably tight.
If $p$ is fixed,
one obtains in this way for each $k$ a $k\times k$ matrix, $A_k(p)$, such that if its 
spectral radius is less than 1, then $P_{\alpha_p}(\log h)<0$ and hence
$\sigma^*$ is an optimal strategy for Ian. This gives a family $(C_k)$ of
sufficient conditions for $\sigma^*$ to be optimal, namely:
\begin{equation}\label{eq:Ck}
\text{If }\rho(A_k(p))<1,\text{ then $\sigma^*$ is an optimal strategy 
for Ian.} \tag{$C_k$}
\end{equation}

\begin{figure}
\includegraphics[width=6cm]{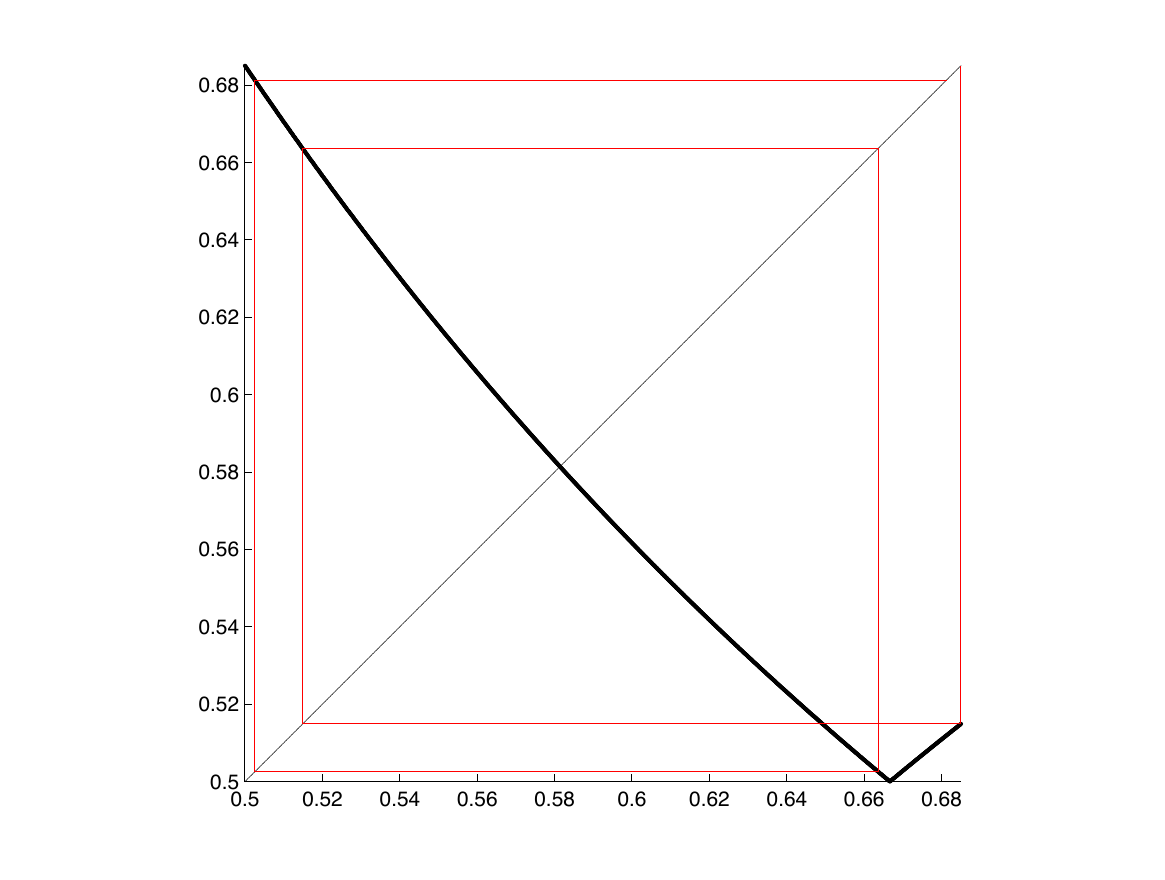}
\includegraphics[width=6cm]{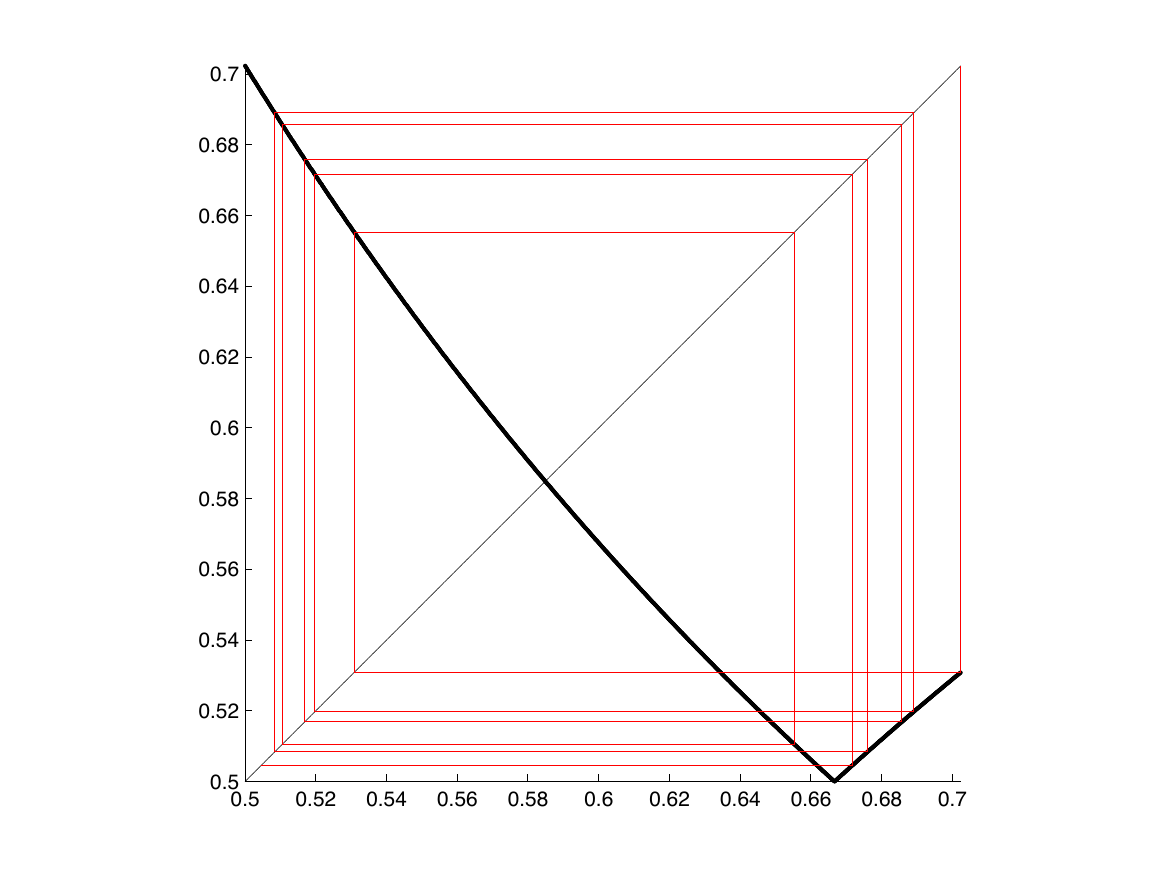}
\includegraphics[width=6cm]{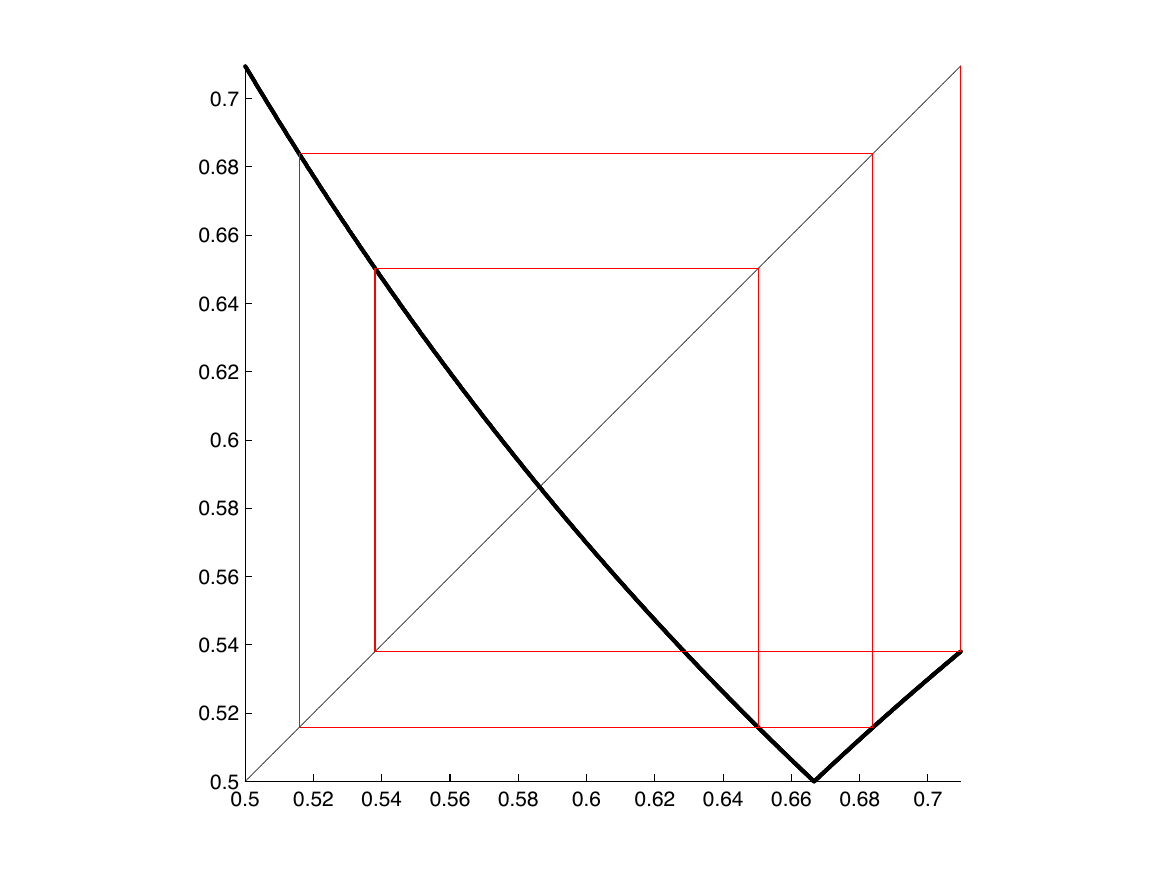}
\includegraphics[width=6cm]{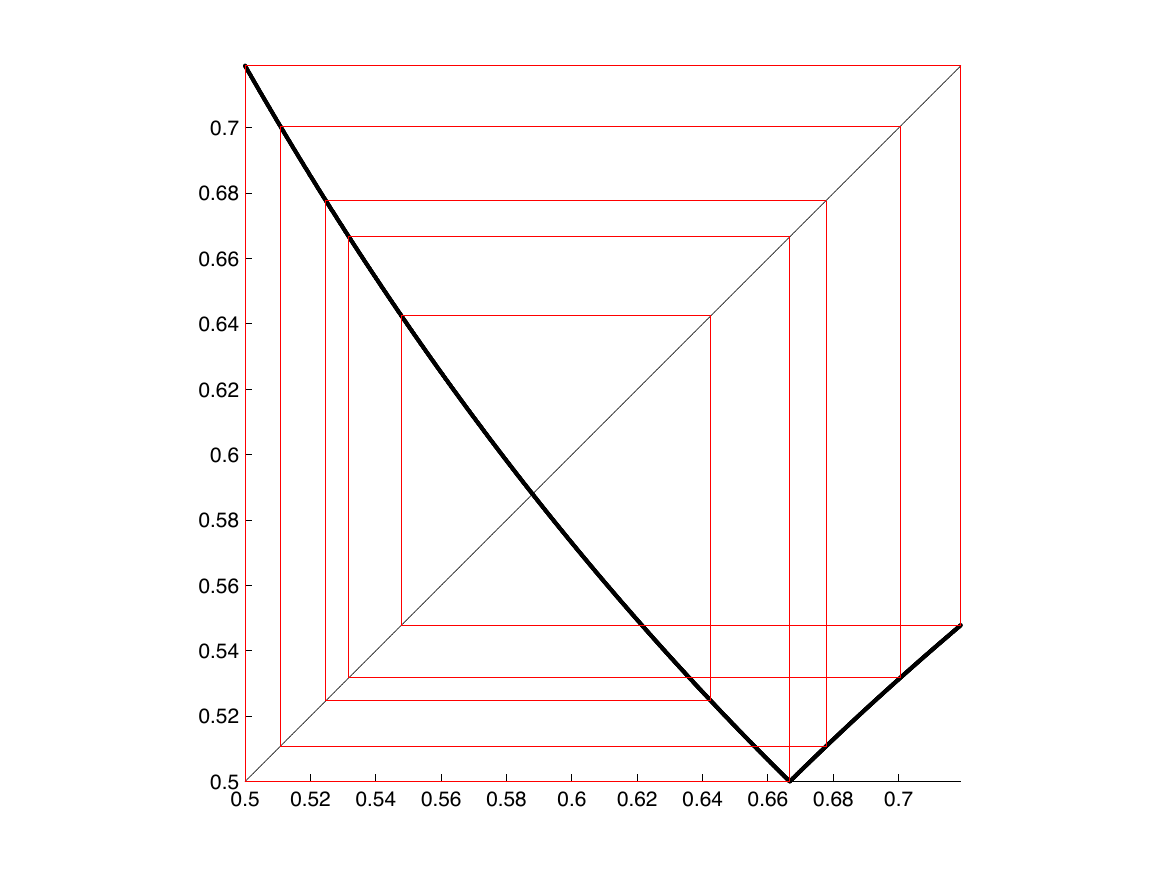}
\caption{The graphs of $\theta \mapsto \alpha(\theta)$ and first points 
of the orbit of $p$, for $p=0.685, p=0.7023\dots, p=0.709\dots$ and 
$p=0.719\dots$. The renormalizablity of $\theta\mapsto\alpha(\theta)$ 
may be seen from the fact that in each of the graphs points to the 
right of the fixed point are mapped to the left of the fixed point 
and vice versa.}
\end{figure}

Proposition \ref{prop:pressureest} and \eqref{eq:Ck} give a way to check that
$P(\log h)<0$ for a single $p$-value. We now obtain estimates on $P(\log h)$
in a range of $p$-values simultaneously.

\subsection{The range (2/3, 0.709636)}

Here, and in the next range, we divide $[\frac 12,p]$ into 9 sub-intervals.
In this range, we check that the following inequalities are satisfied:

$$
\tfrac 12<p_7<p_3<p_5<p_9<p_1<p_2<\tfrac 23<p_6<p_4<p_8<p.
$$

We divide the interval $[\frac 12,p]$ into subintervals $J_0,\ldots,J_8$
as follows: $J_0=[\tfrac 12,p_7]$; $J_1=[p_7,p_3]$; $J_2=[p_3,p_5]$; $J_3=[p_5,p_1]$;
$J_4=[p_1,p_2]$; $J_5=[p_2,p_6]$; $J_6=[p_6,p_4]$; $J_7=[p_4,p_8]$; and $J_8=[p_8,p]$.

The transitions between the intervals are shown in Figure \ref{fig:twoa}.

\begin{figure}
\includegraphics[width=3in]{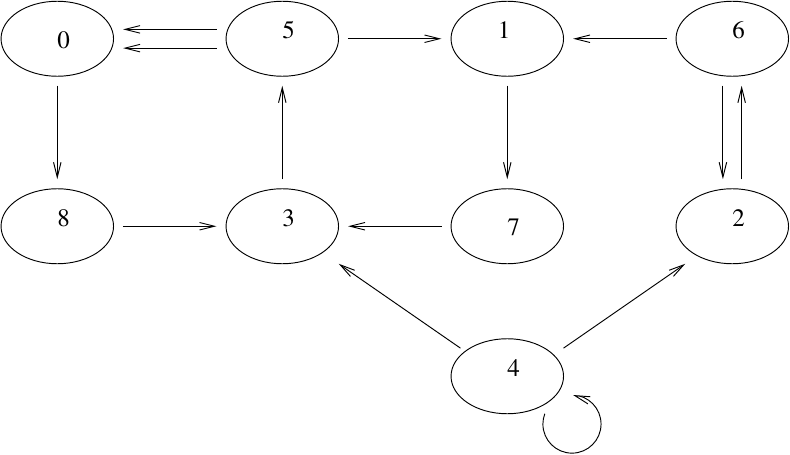}
\caption{Full 9 interval transition diagram for $\frac 23 < p < 0.709637$. 
The double arrow signifies that $m_{50}=2$.}\label{fig:twoa}
\end{figure}

There are three connected components, one (the interval $J_4$ by itself)
with radius $\gamma/p_1$,
one (the intervals $J_2$ and $J_6$) 
with radius $\gamma/\sqrt{p_3p_6}$. Both of these
are less than 1 since $\gamma<\frac 12$. 
The third component is illustrated in Figure \ref{fig:twob}
and consists of two loops of period 4 sharing a common edge.
The spectral radius of this component is the fourth root of the 
sum of the product of the multipliers around the two loops.
That is, the spectral radius of this component is given by
$$
\gamma\left(\frac{1}{p_5p_2}\left(\frac{2}{\frac12p_8}+
\frac{1}{p_7p_4}\right)\right)^{1/4}.
$$

This quantity is less than 1 in the given range. 
\begin{figure}
\includegraphics[width=2.5in]{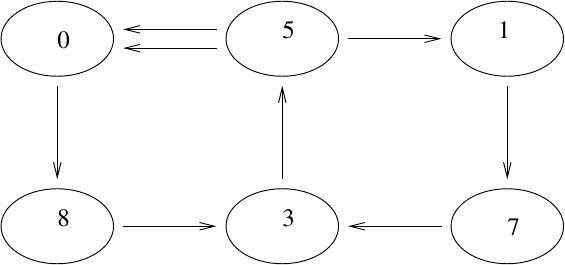}
\caption{Principal component for $\frac 23 < p < 0.709637$}\label{fig:twob}
\end{figure}

Notice that the principal component has period 4 because the original map
is twice renormalizable.

\begin{figure}
\includegraphics[width=2.5in]{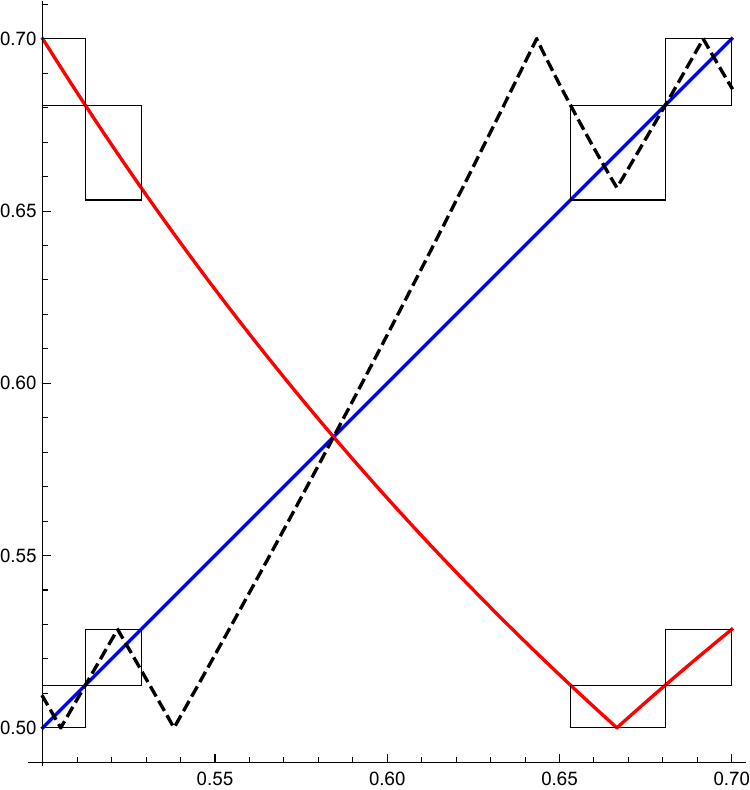}
\caption{Graphs of $\alpha$ (top left to bottom right) and $\alpha^4$ (dashed) for 
$p=0.7$. The map is twice renormalizable, so that there are intervals
$I_1$, $I_2$, $I_3$ and $I_4$ each mapped by $\alpha$ to the next
with $I_1$ containing the critical point. In particular, 
$\alpha^4$ maps each interval to itself. This is illustrated
by the boxes.}\label{fig:renorm4}
\end{figure}

\subsection{The range [0.709637,0.719023]}

In this parameter range, the map is only once renormalizable. 
At 0.709636979, there is a coincidence $p_3=p_5$ (so that all
odd iterates beyond the third coincide; all even iterates beyond the fourth 
coincide). 

The right end point of the interval, 0.7190233023, occurs when $p_9$ hits 
$\frac12$. On the parameter interval $[0.709636979,0.7190233023]$, the
functions $p\mapsto p_i$ are monotone for each $1\le i\le 9$. 
The graphs of the functions do not cross.

In this range, we have
$\frac12<p_9<p_5<p_3<p_7<p_1<p_2<\frac 23<p_8<p_4<p_6<p$.

Again, we use these points (excluding $p_9$ and $\frac 23$) to define a collection of
intervals: $J_0=[\frac12,p_5]$, $J_1=[p_5,p_3]$, 
$J_2=[p_3,p_7]$, $J_3=[p_7,p_1]$, $J_4=[p_1,p_2]$, 
$J_5=[p_2,p_8]$, $J_6=[p_8,p_4]$, $J_7=[p_4,p_6]$ and
$J_8=[p_6,p]$. The transitions are $0\to 8$; $1\to 7$; $2\to 6$; $3\to 5$;
$4\to 2,3,4$; $5\to 0,0,1$; $6\to 0$; $7\to 1,2$; and $8\to 3$ (where 
repeated transitions correspond to values of $m$ that exceed 1). 

This is illustrated in Figure \ref{fig:three}.

\begin{figure}
\includegraphics[width=4in]{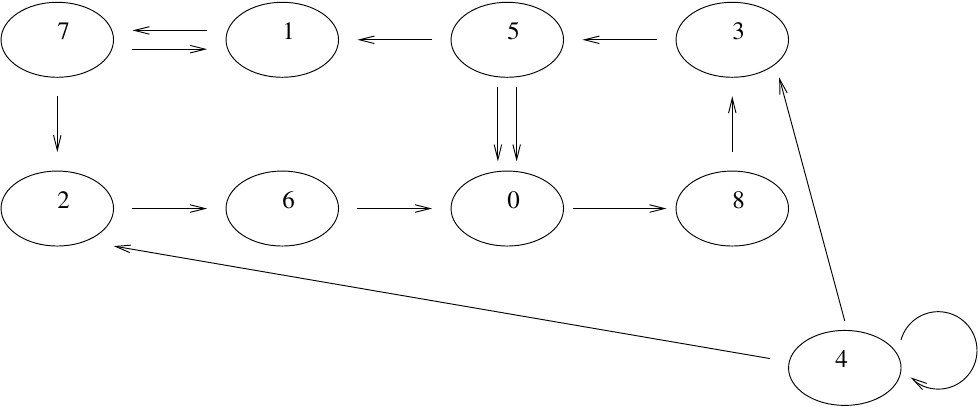}
\caption{The transitions in the range $0.709637 < p < 0.719023$.}
\label{fig:three}
\end{figure}

The single component consisting of $J_4$ always has multiplier less than 1. 
The transition matrix of the principal component is given by

$$
\gamma\begin{pmatrix}
0&0&0&0&0&0&0&2\\
0&0&0&0&0&0&{q_5}&0\\
0&0&0&0&0&{q_3}&0&0\\
0&0&0&0&{q_7}&0&0&0\\
{2}{q_2}&{q_2}&0&0&0&0&0&0\\
{q_8}&0&0&0&0&0&0&0\\
0&{q_4}&{q_4}&0&0&0&0&0\\
0&0&0&{q_6}&0&0&0&0
\end{pmatrix}
$$
where $q_i=1/p_i$.

We check that $q_4$, $q_5$ and $q_8$ are increasing in the parameter range, 
while $q_3$, $q_7$ and $q_6$ are decreasing. 
Substituting the maximum values of each of these quantities in the range
and also using the maximal value of $\gamma$, we obtain a matrix whose 
spectral radius is 0.9773, giving the required bound on the pressure in this range. 

In principle it should be possible to extend by smaller and smaller intervals as long as
the pressure remains negative. For example, the test $(C_{230})$ 
described above shows that the pressure is negative for $p=0.7321$. 
Indeed applying a similar procedure to 10000 randomly chosen $p$-values in the range 
$[0.719,0.732]$ using $(C_k)$ with $k=50,100,150,\ldots,500$ 
shows that $P_{\alpha_p}(\log h)<0$ for each of them.

At this stage, we have proved that the
strategy $\sigma^*$ for Ian and the strategy $\tau$ for Una constructed in
Proposition \ref{prop:suffcond1}
are optimal if $\frac23 \leq p \leq 0.719023$

We define $p_c$ to be the supremum of the set of $t$ such that for each
$p$ satisfying $\frac 12\le p\le t$, $\sigma^*$ is an optimal strategy for Ian.
Combining our results with those of \cite{HRSV}, we have shown 
$p_c\ge 0.719023$. Computer evidence suggests
$0.7321\le p_c\le 0.7322$. We provide an upper bound showing $p_c\le 0.73275300915$
in the next section. We conjecture, based on limited computer experimentation,
that for almost all $p\ge p_c$, $\sigma^*$ is not optimal for Ian.

\section{Beating $\sigma^*$ after the critical point}
\label{sec:better}
For $p$ beyond 0.7322, we suspect that the strategy $\sigma^*$ is often not optimal,
especially when the orbit of $1-p$ comes close to $\frac12$. Indeed, we propose  
strategies --- far from optimal --- which do better than $\sigma^*$ for specific 
values of $p$ ; we prove this claim completely for $\frac34$ (which was an 
explicit open question); we also show the computation for the value 
$p=0.73275300915$. 

Let $p$ be large enough so that we can expect $\sigma^*$ not to be optimal. 
We choose $k_0$ so that $\tilde{\theta} = \alpha^{k_0}(p)$ is  
close to $\frac12$. We also let  $\epsilon>0$ be a small real number. 

We modify slightly $\sigma^*$ to a strategy $\sigma_{k_0,\epsilon}$ in the 
following way: if $\theta \neq \tilde{\theta},1-\tilde{\theta}$,  then Ian
plays following  $\sigma^*$. But if 
$\theta = \tilde{\theta}$ (recall that $\tilde\theta>\frac12$), then Ian  
``perturbs" his reaction by $\epsilon$: he plays $1$ with probability: 
\begin{itemize}
\item $(1-\epsilon)\frac{1-\tilde{\theta}}{\tilde{\theta}}$ if $s=S_1,$
\item $\epsilon$ if $s = S_0$;
\end{itemize}

Meanwhile if $\theta=1-\tilde\theta$, Ian plays 1 with probability
\begin{itemize}
\item $1-\epsilon$ if $s=S_1,$
\item $1- (1-\epsilon) \frac{\tilde{\theta}}{1-\tilde{\theta}} = 
\frac{1-(2-\epsilon)\tilde{\theta}}{1-\tilde{\theta}}$ if $s = S_0,$
\end{itemize}
In the case $\theta=\tilde{\theta}$, the belief is updated as:
\begin{itemize}
\item if Ian plays 1, it becomes : $a_0:= p-\epsilon\gamma$; 
\item if Ian plays 0, it becomes $1-b_0$, where  
$b_0$ is defined to be $2-3p+\frac{\gamma}{\tilde\theta}-\epsilon\gamma
\frac{1-\tilde\theta}{\tilde\theta}$.
\end{itemize}

If $\theta=1-\tilde\theta$, the updates are
\begin{itemize}
\item if Ian plays $1$, it becomes $b_0$
\item if Ian plays $0$, it becomes : $1-a_0$.
\end{itemize}

Notice that $a_0$ is a perturbation of $p$ and $b_0$ is a perturbation 
of $\Phi(\tilde\theta)$. 
The critical aspect in this choice of perturbation of the strategy is that it 
remains U-indifferent: If Una's belief that the system is in state $S_1$ is
$\tilde\theta$, then given the information available to Una, the probability
that the state is $S_0$ and Ian plays 0; and the probability
that the state is $S_1$ and Ian plays 1 are both $(1-\epsilon)(1-\tilde\theta)$.
Similarly if Una's belief is $1-\tilde\theta$, the probabilities
are both $(1-\epsilon)\tilde\theta$. 

It is also greedy except when the belief is $\tilde\theta$ or
$1-\tilde\theta$, in which case the one-step expected gain is 
$(1-\epsilon)\min(\tilde\theta,1-\tilde\theta)$. 
As for $\sigma^*$, Una's belief that the
system is in state $S_1$ evolves as a Markov chain. Since 
Ian's actions do not depend on Una's, one may write down the 
transition probabilities from one state to the next and compute
the expected one-step gain from each state (irrespective of Una's 
choice of move due to the U-irrelevance of the strategy). Hence
is is not hard to obtain an expression for the expected gain of
the perturbed strategy. 

We shall compare the value of  this strategy $\sigma_{k_0,\epsilon}$ 
with the value of $\sigma^*$. We are going to prove 
\begin{lem}
\label{lem:confforbetter}
$v_p(\sigma_{k_0,\epsilon}) > v_p(\sigma^*)$ if and only if 
\begin{equation}\label{eq:desired}
\tilde{\theta} \big( W(\alpha(\tilde{\theta}))  
-  W(b_0)  \big)  
>      (1-\tilde{\theta}) \big(W(a_0)   
-     (1-\epsilon) W(p)  \big), 
\end{equation}
where $\tilde\theta=\alpha^{k_0}(p)$ and $W(\theta) := 
\sum_{n=0}^\infty \prod_{k=0}^{n-1} \alpha^k(\theta)$.
\end{lem} 
In Section \ref{sec:p=3/4}, 
we shall apply this lemma to the case $p=\frac34$ suggested as a 
test case in \cite{HRSV}.

Observe that with the strategy $\sigma_{k_0,\epsilon}$, when 
$\theta=\tilde{\theta}$, the one-step expected 
payoff is a bit smaller than with the strategy $\sigma^*$. However, 
the update of the belief is slightly 
different and one may hope that this new belief puts Ian in a better 
position for the future: in a sufficiently improved position to compensate for the 
loss in the one-step expected payoff. The objective is to show that this 
is possible for some values of $p$.  Note that we make no assertion 
about optimality of the perturbed strategy, but rather show
that irrespective of Una's strategy, the expected gain is larger than
that obtained by playing $\sigma^*$.

For this purpose, we have to find  an expression for the long-term 
expected payoff. Whatever Una plays, the evolution is a Markov chain 
on the beliefs (governed by the random changes of the state and 
the values of his choices). The belief may 
take the values $p$ and $1-p$ and values in the $k_0$ first terms 
of the  orbits of $p$ and $1-p$;
when it reaches $\tilde{\theta}$, it may jump to the values of the belief after 
$\tilde{\theta}$; namely $a_0$ or $1-b_0$ 
and then continue on their orbits for some random time and then go back to 
$1-p$ or $p$. It is convenient to further  assume that neither $\tilde{\theta}$ 
nor $1-\tilde{\theta}$ belong to the orbits of $a_0$ and 
$b_0$ (this is true for all but countably many values of $\epsilon$).
We observe that the symmetry $\theta \mapsto 1-\theta$ does not
affect either the transitions or the payoff so it suffices to follow 
the orbits modulo the symmetry about $\frac12$. 

\begin{figure}
\includegraphics[width=12cm]{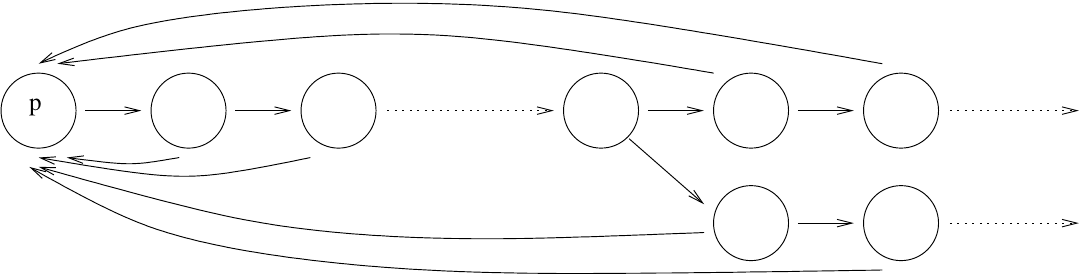}
\caption{Schematic depiction of the (symmetrized) Markov chain. 
At each state other than 
$\tilde\theta$, one choice leads back to the base, and the other goes to the 
right. }\label{graph}
\end{figure}

\subsection{Invariant measure for the Markov Chain}
\begin{proof}[Proof of Lemma \ref{lem:confforbetter}]
Recall $\alpha(\theta):=\max(\Phi(\theta),1-\Phi(\theta))$. 
For $0\le k\le k_0$, let $\Theta_k=\alpha^k(p)$, so that 
$\Theta_{k_0}=\tilde\theta$. Set $a_k=\alpha^k(a_0)$ and $b_k=\alpha^k(b_0)$

We see that Una's belief evolves as a Markov chain on the  countable state space 
$\{ \Theta_k, 0 \leq k \leq k_0 ; a_k, b_k, k \geq 0\}$ 
with transition probabilities: 
\begin{itemize}
\item If $k<k_0$, $\Theta_k \to \Theta_{k+1}$  with probability $\Theta_k$ and 
$\Theta_k \to \Theta_0$ with probability $1-\Theta_k$. 
\item If $k=k_0$, $\Theta_{k_0} \to a_0$ with probability $1-\tilde{\theta}$ 
and $\Theta_{k_0} \to b_0$ with probability $\tilde{\theta}$. 
\item 
For all $k \geq 0$, $a_k \to a_{k+1}$  with 
probability $a_k$ and $a_k \to \Theta_0$ with probability $1-a_k$;
similarly $b_k\to b_{k+1}$ with probability $b_k$ and
$b_k$ to $\Theta_0$ with probability $1-b_k$. 
\end{itemize}

It is straightforward to compute the invariant measure for this chain. We denote 
by $\Pi^\epsilon_s$ the probability of being in $s$ in the perturbed chain
and by $\Pi^0_s$ the probability in the unperturbed chain.

\noindent
For $1 \leq n \leq k_0$
$$
\Pi^\epsilon_{\Theta_n} = \Pi^\epsilon_{\Theta_0} \prod_{k=0}^{n-1} \Theta_k.  
$$
For all $n \geq 0$ 
$$
\Pi^\epsilon_{a_n}=   \Pi^\epsilon_{\Theta_{k_0}} (1- \tilde{\theta})   
\prod_{k=0}^{n-1}a_k \text{\quad and\quad }
\Pi^\epsilon_{b_n} =   \Pi_{\Theta_{k_0}} \tilde{\theta}\prod_{k=0}^{n-1}b_k.    
$$
The Chapman-Kolmogorov equations for $\Pi^\epsilon_{\Theta^0}$ give
\begin{equation}
\label{derniere}
\Pi^\epsilon_{\Theta_0} = \sum_{n=0}^{k_0-1}(1-\Theta_n) 
\Pi^\epsilon_{\Theta_n} +  
\sum_{n=0}^{\infty} (1- a_n) \Pi^\epsilon_{a_n} + 
\sum_{n=0}^{\infty}  (1-b_n) \Pi^\epsilon_{b_n}. 
\end{equation}
Since it is  a probability measure, it also must satisfy : 
\begin{equation}
\label{proba}
\sum_{n=0}^{k_0}  \Pi^\epsilon_{\Theta_n} + \sum_{n=0}^{\infty} \Pi^\epsilon_{a_n} + 
\sum_{n=0}^{\infty}  \Pi^\epsilon_{b_n} = 1. 
\end{equation}
We introduce notation $Q=\prod_{k=0}^{k_0-1}\Theta_k=
\Pi^\epsilon_{\Theta_{k_0}}/\Pi^\epsilon_{\Theta_0}$,
$A=\sum_{n=0}^{k_0}\prod_{j=0}^{n-1}\Theta_j$ and 
$W(\theta)=\sum_{n=0}^\infty \prod_{k=0}^{n-1}\alpha^k(\theta)$.
This latter quantity gives the ratio of the sum of the
weights in the sub-tree rooted at $\theta$ to the weight of $\theta$.
Using this notation, we can write equality \eqref{proba} as
$$
\Pi^\epsilon_{\Theta_0}\left(A + Q
(1- \tilde{\theta}) W(a_0) + Q \tilde{\theta} W(b_0)\right)  = 1.
$$
Hence
$$
\Pi^\epsilon_{\Theta_0} = \left[   A   +  
Q\big((1-\tilde\theta)W(a_0)+\tilde\theta W(b_0)\big)
\right]^{-1}.
$$
Similarly, $\Pi_{\Theta_0}^0=\left[ A+QW(\alpha(\tilde\theta)) \right]^{-1}.$

\subsection{Expected payoff}
The expected payoff can be written as the sum of the expected payoff (given the state)
weighted by the probability of the state; namely, 
$$
v_p(\sigma_{k_0,\epsilon}) = \sum_{n=0}^{k_0-1}  (1- \Theta_n) \Pi^\epsilon_{\Theta_n} + 
(1-\epsilon)(1-\tilde{\theta}) \Pi^\epsilon_{\Theta_{k_0}} +  
\sum_{n=0}^{\infty} (1-a_n) \Pi^\epsilon_{a_n} +  
\sum_{n=0}^{\infty} (1-b_n) \Pi^\epsilon_{b_n}.   
$$
Using \eqref{derniere}, we obtain
$$
v_p(\sigma_{k_0,\epsilon})  =  \Pi^\epsilon_0 +  
(1-\epsilon)(1- \tilde{\theta}) \Pi^\epsilon_{k_0} = 
\Pi^\epsilon_0\big( 1 + (1-\epsilon) (1-\tilde{\theta}) Q\big). 
$$
We want to show that for well chosen $p$, $k_0$ and $\epsilon$, 
$v_p(\sigma_{k_0,\epsilon})  > v_p(\sigma^*)$.  
We recall that $v_p(\sigma^*) = \Pi^0_{\Theta_0}=1/W(p)$. 

We now have
\begin{align*}
&
v_p(\sigma_{k_0,\epsilon})-v_p(\sigma^*)=
\Pi^\epsilon_{\Theta_0}(1+(1-\epsilon)(1-\tilde\theta)Q)-\Pi^0_{\Theta_0}\\
&=\Pi^0_{\Theta_0}\Pi^\epsilon_{\theta_0}
\left(W(p)(1+(1-\epsilon)(1-\tilde\theta)Q)-
( A+Q((1-\tilde\theta)W(a_0)+\tilde\theta W(b_0)))\right)
\end{align*}
Since $W(p)=A+Q\tilde\theta W(\alpha(\tilde\theta))$, we obtain
\begin{equation*}
\begin{split}
&
v_p(\sigma_{k_0,\epsilon})-v_p(\sigma^*)
\\
&=Q\Pi^0_{\Theta_0}\Pi^\epsilon_{\Theta_0}
\left(\tilde\theta (W(\alpha(\tilde\theta))-W(b_0))-
(1-\tilde\theta)(W(a_0)-(1-\epsilon)W(p))\right),
\end{split}
\end{equation*}
completing the proof of Lemma \ref{lem:confforbetter}.
\end{proof}

\subsection{The case  $p$=3/4}\label{sec:p=3/4}
When $p$ takes the value $\frac 34$, 
the symbolic dynamics of $p$ starts with $11010101$ and
$\alpha^7(p) = 1137/2244 \approx 0.5165...$. We shall set $k_0=7$ and 
$\tilde{\theta} =  1137/2244$.  

Next we estimate $W(\theta)$ for the relevant values of $\theta$. 
First we do it for $p$ and for $\alpha(\tilde{\theta})$. 
Recall that $W(\theta) = \sum_{n=0}^{\infty} \prod_{k=0}^{n-1} \alpha^k(\theta)$.
The general term is positive. As soon as $k \geq 1$, $\frac{1}{2} \leq \Theta_k \leq p$.
Hence, the remainder of the sequence is bounded by 
\begin{align*}
\sum_{n \geq N}  \prod_{k=0}^{n-1} \alpha^k(\theta) 
&\leq  \left(\prod_{k=0}^{N-1} \alpha^k(\theta)\right)  \sum_{n \geq 0} p^{n} \\
&\leq \frac{  \left(\prod_{k=0}^{N-1} \alpha^k(\theta)\right) }{1-p}  
\leq \frac{p^N}{1-p} \leq 4\left(\tfrac{3}{4} \right)^{N}.
\end{align*}

We do the computation with $N=50$, so the bound  on the error is smaller than 
$10^{-10}$ (and the obvious bound $p^{-N}$ is itself of order $10^{-7}$).  
We obtain with this approximation $W(p) \approx 2.8354$ and 
$W(\alpha(\tilde{\theta})) \approx  2.7432$. 

Then numerical experimentation (see Figure \ref{fig:epsilon})
suggests taking $\epsilon = 0.01$.
\begin{figure}
\includegraphics[width=10cm]{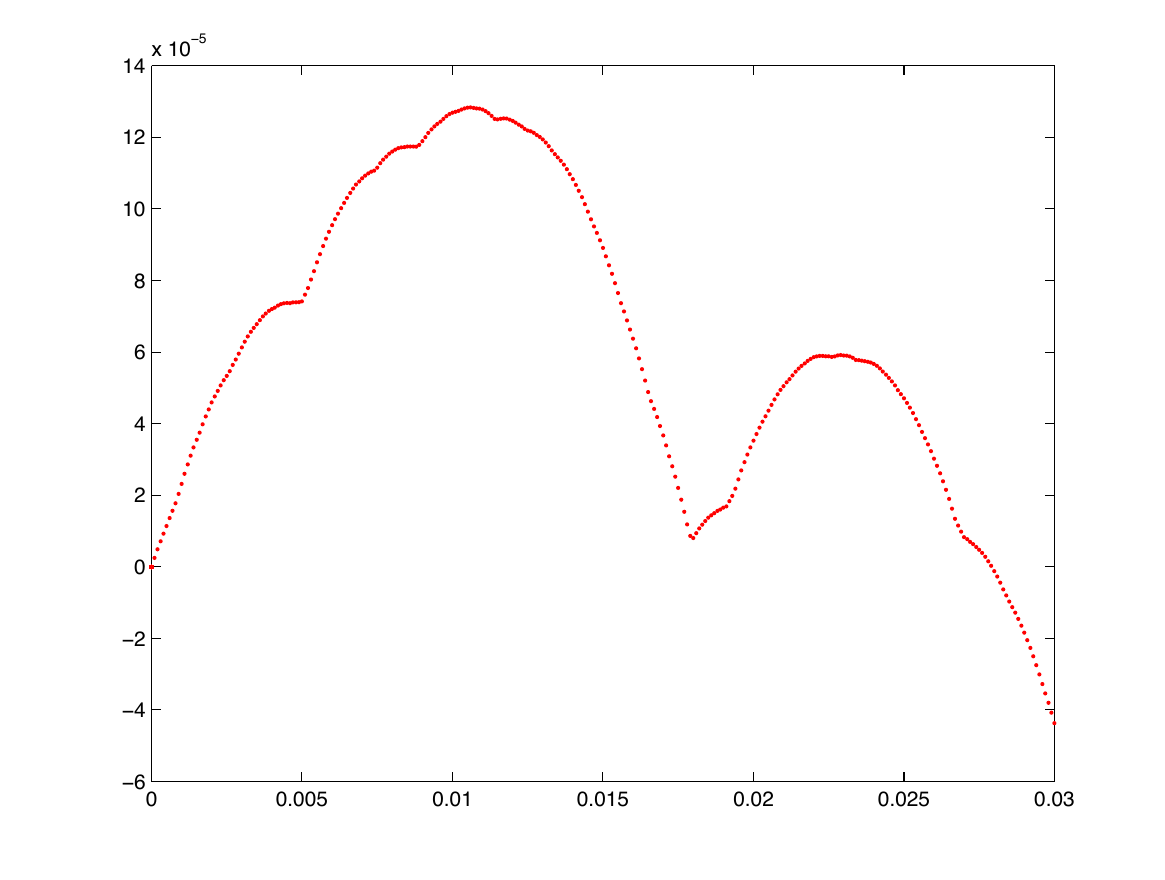}
\caption{Numerical approximation of the graph of the left side of 
\eqref{eq:deltagain} (vertical axis) in the case $p=\frac 34$
as a function of $\epsilon$ (horizontal axis).} \label{fig:epsilon}
\end{figure}
For this value of $\epsilon$, we also compute $W(b_0) 
\approx 2.7305$ and $W(a_0) \approx   2.8203$. 
This is  sharp enough to see the difference between  
$$
(1-\tilde{\theta}) \big(W(a_0) - (1-\epsilon) W(p) \big)
\approx     0.0064
$$ 
and 
$$ 
\tilde{\theta}  \big( W(\alpha(\tilde{\theta}) - W(b_0) )\big)\approx 0.0065. 
$$
We conclude that 
\begin{equation}\label{eq:deltagain}
\tilde{\theta} \big( W(\alpha(\tilde{\theta}))  
-  W(b_0)  \big)  
-      (1-\tilde{\theta}) \big(W(a_0)-(1-\epsilon) W(p))  \big) >    10^{-4}, 
\end{equation}
so that, by Lemma~\ref{lem:confforbetter}, we have shown 
that $\sigma^*$ is not optimal for $p=\tfrac 34$. The expected 
payoff of the alternative strategy can be computed: 
we obtain $v_{\frac{3}{4}}(\sigma^*)=0.35267910...$ and 
$v_{\frac{3}{4}}(\sigma_{7,0.01}) = 0.35267964...$, showing  a difference 
between the values  of 
$$
v_{\frac{3}{4}}(\sigma_{7,0.01}) -   v_{\frac{3}{4}}(\sigma^*) \approx  5\times 10^{-7}.
$$
\subsection{The case  $p$=0.73275300915}
By trial and error, we located a value of $p$ slightly above the
conjectured critical point $p_c\approx 0.7321$ for which
 $\sigma^*$ is not optimal. 
Computations (using the Mathematica package with 200 digit accuracy)
with $p=0.73275300915$, $k_0 = 57$ and $\epsilon=0.0002$ 
show that $\tilde\theta\approx 0.50000194899$, 
$v_p(\sigma_{57,0.0002})\approx 0.361469540454503987436\mathbf{65121}$
and $v_p(\sigma^*)\approx 0.361469540454503987436\mathbf{10381}$, so that the
gain of the perturbed strategy is larger by approximately $5.47\times 10^{-22}$.
This concludes the proof of Theorem~\ref{th:result}. 

\section{Overview}

We hope that ideas from this paper may find wider application in the theory
of repeated games. We identify a couple of factors that play important roles
in our analysis:

\begin{description}
\item[Renewal]
The directed graph describing the evolution of Una's beliefs has a very
simple structure (see Figure \ref{fig:infstate}). Any time that Ian's move is
aligned with Una's belief, her belief returns to the base of the tower.
This renewal structure vastly simplifies computations.
\item[Complexity and Contraction]
Our construction of Una's best response to $\sigma^*$ was based on solving
a system of linear equations \eqref{eq:GHeqs}
relating the values of $V$ before and after Ian's move. 
The contraction properties of the matrices guaranteed the existence of a fixed
point of $\Lax$. Our method depended also on getting detailed information about the 
fixed point. The discontinuity of $\Phi$ at $\frac 12$ led to discontinuities
of $V$ at $\frac 12$. These are propagated by
\eqref{eq:GHeqs} to preimages of $\frac 12$ under $\Phi$. A key role was played in the 
argument by the fact that the jumps at the discontinuity points were
all of the same sign and summable (the summability ensured that
the fixed point was of pure jump type and the sign condition ensured that the fixed
point was monotonic). That the sign was constant appears to be
a fortunate accident. The summability can be traced to the complexity of $\Phi$. 
When $p<\frac 23$, there are no preimages of $\frac 12$. As $p$ increases,
the complexity of the system (the topological entropy) increases. This quantity
measures the exponential growth rate of the number of preimages. The pressure
measures a combination of the number of preimages with the size of the discontinuity
at each. 
\end{description}

We thank the two referees for an extremely careful and constructive reading
of the paper, as well as for the numerous suggestions for improvements.

\bibliographystyle{abbrv}
\bibliography{games}

\end{document}